\documentclass{amsart}
\usepackage{moreverb,url}
\usepackage{amssymb,amsmath,amsfonts,mathrsfs}
\usepackage{bm}
\usepackage{float}
\usepackage{verbatim}
\usepackage[square,numbers]{natbib}
\usepackage{stmaryrd}
\usepackage{fancyhdr}
\usepackage{syntax,etoolbox}
\usepackage{graphicx}
\usepackage{grffile}
\usepackage[dvipsnames]{xcolor}
\usepackage{subcaption}
\usepackage[font=scriptsize, labelfont=bf]{caption}
\usepackage[square,numbers]{natbib}
\usepackage[left=2.5cm, right=2.5cm, bottom=3cm]{geometry}
\usepackage{lineno}
\allowdisplaybreaks

\usepackage{hyperref}
\hypersetup{colorlinks=true,allcolors=blue}
\usepackage{hypcap}
\usepackage{leftidx}
\usepackage{xcolor}
\usepackage[symbols,nogroupskip,sort=use]{glossaries-extra} 
\makeindex
\makenoidxglossaries
\newglossaryentry{Omega}{
	name=\ensuremath{\Omega},
	description={open set or domain in $\mathbb{R}^3$},
	type=symbols
}
\newglossaryentry{mathbbL2}{
	name=\ensuremath{\mathbb{L}^2(\Omega)},
	description={space of square-integrable symmetric tensors defined over $\Omega$},
	type=symbols
}
\newglossaryentry{Lp0TX}{
	name=\ensuremath{L^p(0,T;X)},
	description={space of $X$-valued functions that are $L^p$ integrable in $(0,T)$},
	type=symbols
}
\newglossaryentry{Xast}{
	name=\ensuremath{X^\ast},
	description={dual space of the Banach space $X$},
	type=symbols
}
\newglossaryentry{duality}{
	name=\ensuremath{\langle \cdot, \cdot \rangle_{X^\ast, X}},
	description={duality pair between $X^\ast$ and $X$},
	type=symbols
}
\newglossaryentry{C0TX}{
	name=\ensuremath{\mathcal{C}^0(0,T;X)},
	description={space of $X$-valued functions that are continuous in $[0,T]$},
	type=symbols
}
\newglossaryentry{dualityC}{
	name=\ensuremath{\langle\langle\cdot,\cdot\rangle\rangle_{X^\ast,X}},
	description={duality pair between $(\mathcal{C}^0([0,T];X))^\ast$ and $\mathcal{C}^0([0,T];X)$},
	type=symbols
}
\newglossaryentry{M0TX}{
	name=\ensuremath{\mathcal{M}([0,T];X^\ast)},
	description={space of $X^\ast$-valued measures defined over the compact interval $[0,T]$},
	type=symbols
}
\newglossaryentry{doteta}{
	name=\ensuremath{\dot{\eta}},
	description={first weak derivative with respect to the time variable of the function $\eta:(0,T)\to \mathbb{R}$},
	type=symbols
}
\newglossaryentry{ddoteta}{
	name=\ensuremath{\ddot{\eta}},
	description={second weak derivative with respect to the time variable of the function $\eta:(0,T)\to \mathbb{R}$},
	type=symbols
}
\newglossaryentry{dotbmeta}{
	name=\ensuremath{\dot{\bm{\eta}}},
	description={first weak derivative with respect to the time variable of the function $\bm{\eta}:(0,T)\to X$},
	type=symbols
}
\newglossaryentry{ddotbmeta}{
	name=\ensuremath{\ddot{\bm{\eta}}},
	description={second weak derivative with respect to the time variable of the function $\bm{\eta}:(0,T)\to X$},
	type=symbols
}
\newglossaryentry{Gamma}{
	name=\ensuremath{\Gamma},
	description={boundary of $\Omega$},
	type=symbols
}
\newglossaryentry{dx}{
	name=\ensuremath{\mathrm{d}x},
	description={volume element of $\Omega$},
	type=symbols
}
\newglossaryentry{dGamma}{
	name=\ensuremath{\mathrm{d}\Gamma},
	description={area element along $\Gamma$},
	type=symbols
}
\newglossaryentry{Gamma0}{
	name=\ensuremath{\Gamma_0},
	description={portion of $\Gamma$ where homogeneous Dirichlet boundary conditions of place are imposed},
	type=symbols
}
\newglossaryentry{Gamma1}{
	name=\ensuremath{\Gamma_1},
	description={complementary of $\Gamma_0$ on $\Gamma$},
	type=symbols
}
\newglossaryentry{E3}{
	name=\ensuremath{\mathbb{E}^3},
	description={real three-dimensional affine Euclidean space},
	type=symbols
}
\newglossaryentry{ei}{
	name=\ensuremath{\bm{e}^i},
	description={vectors constituting an orthonormal basis for $\mathbb{E}^3$},
	type=symbols
}
\newglossaryentry{adotb}{
	name=\ensuremath{\bm{a}\cdot\bm{b}},
	description={Euclidean inner product between the vectors $\bm{a}$ and $\bm{b}$ in $\mathbb{E}^3$},
	type=symbols
}
\newglossaryentry{norma}{
	name=\ensuremath{|\bm{a}|},
	description={Euclidean norm of the vector $\bm{a}\in\mathbb{R}^3$},
	type=symbols
}
\newglossaryentry{deltaij}{
	name=\ensuremath{\delta^{ij}},
	description={Kronecker symbol},
	type=symbols
}
\newglossaryentry{lambda}{
	name=\ensuremath{\lambda},
	description={Lam\'{e} constant with the property of being non-negative},
	type=symbols
}
\newglossaryentry{mu}{
	name=\ensuremath{\mu},
	description={Lam\'{e} constant with the property of being strictly positive},
	type=symbols
}
\newglossaryentry{theta}{
	name=\ensuremath{\theta},
	description={viscosity constant with the property of being non-negative},
	type=symbols
}
\newglossaryentry{xi}{
	name=\ensuremath{\xi},
	description={viscosity constant with the property of being strictly positive},
	type=symbols
}
\newglossaryentry{rho}{
	name=\ensuremath{\rho},
	description={mass density},
	type=symbols
}
\newglossaryentry{a.e.}{
	name=\ensuremath{\textup{a.e.}},
	description={almost everywhere},
	type=symbols
}
\newglossaryentry{a.a.}{
	name=\ensuremath{\textup{a.a.}},
	description={almost all},
	type=symbols
}
\newglossaryentry{VOmega}{
	name=\ensuremath{\bm{V}(\Omega)},
	description={$\{\bm{v}=(v_i)\in \bm{H}^1(\Omega); \bm{v}={\bf{0}} \dd\Gamma-\textup{a.e. on } \Gamma_0\}$},
	type=symbols
}
\newglossaryentry{tr}{
	name=\ensuremath{\textup{tr}},
	description={trace operator from $\bm{V}(\Omega)$ onto $\bm{L}^2(\Gamma)$},
	type=symbols
}
\newglossaryentry{A}{
	name=\ensuremath{A^{ijk\ell}},
	description={fourth order three-dimensional elasticity tensor in Cartesian coordinates},
	type=symbols
}
\newglossaryentry{B}{
	name=\ensuremath{B^{ijk\ell}},
	description={fourth order three-dimensional viscosity tensor in Cartesian coordinates},
	type=symbols
}
\newglossaryentry{eij}{
	name=\ensuremath{e_{i\|j}},
	description={linearised change of metric tensor in Cartesian coordinates},
	type=symbols
}
\newglossaryentry{c0}{
	name=\ensuremath{c_0},
	description={constant of the Korn's inequality (Theorem~\ref{KornCart})},
	type=symbols
}
\newglossaryentry{AA}{
	name=\ensuremath{\mathcal{A}},
	description={elasticity operator in the divergence form},
	type=symbols
}
\newglossaryentry{Ce1}{
	name=\ensuremath{C_e^{(1)}},
	description={uniform positive-definiteness constant for the fourth order three-dimensional elasticity tensor $\{A^{ijk\ell}\}$},
	type=symbols
}
\newglossaryentry{BB}{
	name=\ensuremath{\mathcal{B}},
	description={viscosity operator in the divergence form},
	type=symbols
}
\newglossaryentry{Ce2}{
	name=\ensuremath{C_e^{(2)}},
	description={uniform positive-definiteness constant for the fourth order three-dimensional viscosity tensor $\{B^{ijk\ell}\}$},
	type=symbols
}
\newglossaryentry{q}{
	name=\ensuremath{\bm{q}},
	description={given unit-norm vector in $\mathbb{E}^3$},
	type=symbols
}
\newglossaryentry{I}{
	name=\ensuremath{\bm{I}},
	description={undeformed reference configuration of the linearly viscoelastic body},
	type=symbols
}
\newglossaryentry{H}{
	name=\ensuremath{\mathbb{H}},
		description={orthogonal complement of $\bm{q}$ and half-space where the viscoelastic body has to remain confined},
	type=symbols
}
\newglossaryentry{UOmega}{
	name=\ensuremath{\bm{U}(\Omega)},
	description={set of admissible displacements},
	type=symbols
}
\newglossaryentry{f}{
	name=\ensuremath{\bm{f}},
	description={applied body force},
	type=symbols
}
\newglossaryentry{u0}{
	name=\ensuremath{\bm{u}_0},
	description={initial displacement},
	type=symbols
}
\newglossaryentry{u1}{
	name=\ensuremath{\bm{u}_1},
	description={initial velocity},
	type=symbols
}
\newglossaryentry{kappa}{
	name=\ensuremath{\kappa},
	description={positive penalty parameter intended to tend to zero},
	type=symbols
}
\newglossaryentry{f-}{
	name=\ensuremath{\{f\}^{-}},
	description={$-\min\{f,0\}$},
	type=symbols
}
\newglossaryentry{N}{
	name=\ensuremath{\mathcal{N}},
	description={distributional formulation of the penalty term},
	type=symbols
}
\newglossaryentry{trstar}{
	name=\ensuremath{\textup{tr}^\star},
	description={Banach adjoint of $\textup{tr}$},
	type=symbols
}
\newglossaryentry{tildeN}{
	name=\ensuremath{\tilde{\mathcal{N}}},
	description={time-dependent version of $\mathcal{N}$},
	type=symbols
}
\newglossaryentry{tildeA}{
	name=\ensuremath{\tilde{\mathcal{A}}},
	description={time-dependent version of $\mathcal{A}$},
	type=symbols
}
\newglossaryentry{tildeB}{
	name=\ensuremath{\tilde{\mathcal{B}}},
	description={time-dependent version of $\mathcal{B}$},
	type=symbols
}
\newglossaryentry{AC}{
	name=\ensuremath{AC([0,T])},
	description={space of absolutely continuous functions over $[0,T]$},
	type=symbols
}

\newtheorem{theorem}{Theorem}[section]
\newtheorem{lemma}{Lemma}[section]

\newtheorem{innercustomgeneric}{\customgenericname}
\providecommand{\customgenericname}{}
\newcommand{\newcustomproblem}[2]{%
	\newenvironment{#1}[1]
	{%
		\renewcommand\customgenericname{#2}%
		\renewcommand\theinnercustomgeneric{##1}%
		\innercustomgeneric
	}
	{\endinnercustomgeneric}
}

\newcustomproblem{customprob}{Problem}

\newcommand*{\bqed}{\hfill\ensuremath{\blacksquare}}%
\newcommand{\wsc}{\overset{\ast}{\rightharpoonup}}%

\def\dd{\, \mathrm{d}}

\begin{document}
	
	\today
	
	
	\title[Time-dependent Signorini-type problems in viscoelasticity]{Existence of solutions for time-dependent Signorini-type problems in linearised viscoelasticity}
	
	
	\author{Paolo Piersanti}
	\address{School of Science and Engineering, The Chinese University of Hong Kong (Shenzhen), 2001 Longxiang Blvd., Longgang District, Shenzhen, China}
	\email{ppiersanti@cuhk.edu.cn}

	\begin{abstract}
		In this paper we establish the existence of solutions for a model describing the evolution of a linearly viscoelastic body which is constrained to remain confined in a prescribed half-space. The confinement condition under consideration is of Signorini type, and is given over the boundary of the linearly viscoelastic body under consideration. We show that one such variational problem admits solutions and we coin a novel concept of solution which, differently from the available literature, is valid even in the case where the viscoelastic body starts its motion in contact with the obstacle. Additionally, under additional assumptions on the constituting material, we show that when the applied body force is lifted the deformed linearly viscoelastic body returns to its rest position at an exponential rate of decay.
		
		
		\smallskip
		\noindent \textbf{MSC~2020.} 35L85, 47H07, 74B10.
	\end{abstract}
	
	\maketitle
	
\section{Introduction} \label{Sec:0}

The contact of elastic and viscoelastic materials with rigid obstacles arises in many applicative fields, including biomedical engineering, aerospace design, and materials science. For instance, the motion of aortic heart valves - modelled as linearly elastic shells constrained within a spatial domain to avoid collisions - has inspired mathematical frameworks for obstacle problems in elasticity \cite{HanSofonea2002,HHNL1988,Rodri2018,SHS06}. Similarly, the confinement of deformable structures like polymer films or viscoelastic membranes in prescribed half-spaces is critical for applications ranging from soft robotics to biomechanical implants \cite{CambridgeVisco2013}. Recently, the theory of obstacle problems was used to model ice melting sheets~\cite{FRS24,JB12,PieTem2023}.

The mechanical deformation of linearly viscoelastic bodies is classically modelled using energy functionals from linearized elasticity theory \cite{Ciarlet1988,HanSofonea2002}. Recent advances extend these principles to dynamic and time-dependent settings. For example, Bock and Jaru\v{s}ek analysed unilateral dynamic contact problems for viscoelastic Reissner-Mindlin plates \cite{BJ11}, while \cite{BockJarSil2016} established existence results for thermoelastic von K\'{a}rm\'{a}n plates vibrating against rigid obstacles. A key challenge in such problems lies in reconciling hyperbolic variational inequalities for displacements with memory-dependent constitutive laws, as seen in quasi-static viscoelasticity with self-contact \cite{KR20,CGK24}. The latter introduces novel solution concepts distinct from classical frameworks \cite{Bre72,Lions1969}, emphasising admissible forces that preclude terminal-phase collisions.

Recent studies highlight unresolved questions in viscoelastic obstacle problems, such as the interplay between memory kernel decay rates and solution regularity \cite{ViscoQuestions2022}. These align with broader inquiries into non-local integro-differential equations governing viscoelastic dynamics, where semigroup methods and spectral analysis remain pivotal \cite{ViscoQuestions2022}. Meanwhile, confinement conditions in linearized elasticity, particularly for three-dimensional bodies in half-spaces, demand sophisticated Sobolev-space analyses~\cite{EBJ05}.

For what concerns the study of the dynamics linearly elastic bodies subjected to remaining confined in a prescribed half-space, we refer to~\cite{Pie2020}. The existence result in this paper is established by resorting to vector-valued measures in the sense of Dinculeanu. The latter approach requires some strong assumptions, which we aim to overcome in the present paper.

In this paper, we formulate a three-dimensional obstacle problem for a time-dependent linearly viscoelastic body constrained to a half-space. To the best of our knowledge, the current literature does not address the study of time-dependent obstacle problems for three-dimensional viscoelastic bodies subjected to a confinement condition Signorini-type. For completeness, we mention the textbook by Han \& Sofonea~\cite{HanSofonea2002} where quasi-static viscoelastic problems are studied by means of a different concept of solution from the one employed here. Additionally, we mention the recent papers~\cite{BD16,BKK15,ZBC25} where the quasi-static Signorini problems in viscoelasticity are considered from the analytical and numerical point of view.

Our approach synthesizes ideas from \cite{CGK24}, where forces appear to avoid terminal-phase contact, and the geometric confinement frameworks of \cite{CiaPie2018b}. Critical to establish the existence of solutions for the problem considered here is the compactness for the \emph{time-dependent version} of the trace operator (Lemma~\ref{lem:tr}), which appears to be a novel result to the best of our knowledge.
We also mention the recent paper~\cite{Han26}, where the authors consider a Signorini problem in linearised viscoelasticity with a contact condition for both the displacement and the velocity. In this contribution, the authors penalise the problem and obtain estimates depending on a parameter.

In addition to the latter, the novelties introduced in this paper are:
\begin{itemize}
	\item[$(1)$] The derivation of a sound concept of solution for a time-dependent obstacle problem of Signorini type in linearised viscoelasticity, which aligns to the classical concept of solution for obstacle problems for elliptic models;
	\item[$(2)$] The methodology here presented, differently from~\cite{CGK24}, takes an initial condition for the velocity into consideration. Additionally, our proposed methodology does not take into account the terminal condition implicitly assumed in~\cite{CGK24};
	\item[$(3)$] Our methodology takes into account a confinement condition of Signorini type rather than the Ciarlet-Ne\v{c}as condition (see~\cite{CN87}) considered in~\cite{CGK24};
	\item[$(4)$] Differently from~\cite{BJ10,BJ11,BJ15,BockJarSil2016}, the methodology here presented applies to obstacle problems formulated in terms of vectorial magnitudes;
	\item[$(5)$] By contrast with~\cite{Han26}, we only impose the constraint on the displacement and we consider non-zero initial data and we do not assume additional regularity for the solution;
	\item[$(6)$] By contrast with the recent works~\cite{CiaPie2018b,Pie2020}, our methodology has the advantage of applying to bodies which start their motion in contact with the obstacle. Moreover, all the extra assumptions made in~\cite{Pie2020} are dropped, suggesting that viscoelastic models are more suitable for studying the dynamics of elastic bodies making contact with obstacles during their motion. This statement is corroborated by the conclusion derived in Theorem~\ref{th:decay}, where we show that when the applied body force is released the body tends to return to its original reference configuration.
\end{itemize}

The concept of solution for the governing model is recovered by two layer of approximations: in the first layer, we \emph{relax} the constraint by penalising the classical energy in linearised viscoelasticity, while in the second layer we discretise the \emph{relaxed} problem by Galerkin method.

The paper is divided into four sections, including this one. In section~\ref{Sec:1} we present the main notation from elasticity and differential geometry. In section~\ref{Sec:2} we formulate the obstacle problem for a three-dimensional viscoelastic body. In section~\ref{Sec:3}, we recover the existence of solutions for the model under consideration by means of compactness methods.

\section{Geometrical preliminaries} \label{Sec:1}

For details about the classical notions of differential geometry recalled in this section, see, e.g.~\cite{Ciarlet2000} or \cite{Ciarlet2005}.

Latin indices, except when they are used for indexing sequences, take their values in the set $\{1,2,3\}$, and the summation convention with respect to repeated indices is systematically used in conjunction with this rule.

For an open subset $\gls{Omega} \subset \mathbb{R}^3$, the notations $L^2(\Omega)$ and $H^1(\Omega)$ refer to the standard Lebesgue and Sobolev spaces, respectively. The notation $\mathcal D(\Omega)$ indicates the space of functions that are infinitely differentiable on $\Omega$ and have compact support within $\Omega$. The notation $|\cdot|_X$ denotes the norm in a vector space $X$. Spaces of vector-valued functions are represented by boldface letters like for instance $\bm{L}^2(\Omega)$ or $\bm{H}^1(\Omega)$. blueSpaces of symmetric tensors are denoted in blackboard bold letters like for instance \gls{mathbbL2}.

Lebesgue-Bochner spaces (see, e.g., \cite{Leoni2017}) are represented by the notation \gls{Lp0TX}, where $1 \le p \le \infty$, $T>0$ and $X$ is a Banach space satisfying the Radon-Nikodym property. The notation \gls{Xast} refers to the dual space of a vector space $X$, and the notation \gls{duality} denotes the duality pairing between $X^\ast$ and $X$. The notation \gls{C0TX} designates the space of continuous $X$-valued functions defined over the compact interval $[0,T]$.

The notations \gls{doteta} and \gls{ddoteta} represent the first weak derivative with respect to $t \in (0,T)$ and the second weak derivative with respect to $t \in (0,T)$ of a scalar function $\eta$ defined over the interval $(0,T)$. The notations \gls{dotbmeta} and \gls{ddotbmeta} denote the first weak derivative with respect to $t \in (0,T)$ and the second weak derivative with respect to $t \in (0,T)$ of a vector field $\bm{\eta}$ defined over the interval $(0,T)$.

A \emph{Lipschitz domain} $\Omega \subset \mathbb{R}^3$ is a non-empty, open, bounded, and connected subset with a Lipschitz continuous boundary \gls{Gamma}, where the set $\Omega$ is locally on the same side of $\Gamma$. The notation \gls{dx} indicates the \emph{volume element} in $\Omega$, and the symbol \gls{dGamma} indicates the \emph{area element} along $\Gamma$. For more details about this definition see, for instance, Section~8.2 in~\cite{Cia25}.
Finally, let $\Gamma=\gls{Gamma0} \cup \gls{Gamma1}$ be a $\dd \Gamma$-measurable portion of the boundary such that $\Gamma_0 \cap \Gamma_1 =\emptyset$ and $\textup{area }\Gamma_0>0$.

As a model of the three-dimensional ``physical space'' $\mathbb{R}^3$, we consider a \emph{real three-dimensional affine Euclidean space}, which is defined by selecting a point $O$ as the \emph{origin} and associating it with a \emph{real three-dimensional Euclidean space}, denoted \gls{E3}. We equip $\mathbb{E}^3$ with an \emph{orthonormal basis} consisting of three vectors \gls{ei}. The Euclidean inner product of two elements $\bm{a}$ and $\bm{b}$ in $\mathbb{E}^3$ is denoted by \gls{adotb}; the Euclidean norm of any $\bm{a} \in \mathbb{E}^3$ is denoted by \gls{norma}; the Kronecker symbol is denoted by \gls{deltaij}.

The characterization of $\mathbb{R}^3$ as an affine Euclidean space implies that with each point $x \in \mathbb{R}^3$ is associated a uniquely defined vector $\bm{Ox} \in \mathbb{E}^3$. The origin $O \in \mathbb{R}^3$ and the orthonormal vectors $\bm{e}^i \in \mathbb{E}^3$ together form a \emph{Cartesian frame} in $\mathbb{R}^3$, and the three components $x_i$ of the vector $\bm{Ox}$ relative to the basis $\bm{e}^i$ are referred to as the \emph{Cartesian coordinates} of $x \in \mathbb{R}^3$, or the \emph{Cartesian components} of $\bm{Ox} \in \mathbb{E}^3$. Once a Cartesian frame has been established, any point $x \in \mathbb{R}^3$ can be \emph{identified} with the vector $\bm{Ox}=x_i \bm{e}^i \in \mathbb{E}^3$. We then denote $\partial_i=\partial/\partial x_i$.

The set $\overline{\Omega}$ is the \emph{reference configuration} occupied by a \emph{linearly viscoelastic body} in the absence of applied forces. We assume that $\overline{\Omega}$ is in a natural state, meaning the body is stress-free in this configuration. Following \cite{Ciarlet1988}, we also assume that the material is \emph{isotropic}, \emph{homogeneous}, and \emph{linearly viscoelastic}. Under these conditions, the behaviour of the linearly elastic material is governed by its two \emph{Lamé constants} $\gls{lambda} \ge 0$ and $\gls{mu} > 0$. We denote the viscosity constants by $\gls{theta} \ge 0$ and $\gls{xi}>0$. The positive constant $\gls{rho}$ designates the \emph{mass density} of the linearly viscoelastic body per unit volume.

In what follows, ``\gls{a.e.}'' stands for ``almost everywhere'' and ``\gls{a.a.}'' stands for ``almost all''. Define the space 
\begin{equation*}
\gls{VOmega}:=\{\bm{v}=(v_i)\in \bm{H}^1(\Omega); \bm{v}={\bf{0}}\, \dd\Gamma\textup{-a.e. on } \Gamma_0\},
\end{equation*}
and equip it with the norm
\begin{equation*}
\|\bm{v}\|_{\bm{V}(\Omega)}:=\left(\sum_{i} \|v_i\|_{H^1(\Omega)}^2\right)^{1/2}.
\end{equation*}

Note that in the definition of the space $\bm{V}(\Omega)$ we require the trace of any element $\bm{v}$ to vanish on $\Gamma_0$ up to a zero-measure subset of $\Gamma_0$ with respect to the measure of surfaces. The trace operator is denoted and defined by $\gls{tr}:\bm{V}(\Omega)\to \bm{L}^2(\Gamma)$, and we recall that $\textup{tr}$ is linear, continuous and compact~\cite{Nec67}.

Define the fourth order three-dimensional elasticity tensor in Cartesian coordinates and denote its components by \gls{A}.
We recall that the contravariant components of this tensor are defined by (see, e.g., \cite{Ciarlet1988}) 
\begin{equation*}
A^{ijk\ell}:=\lambda \delta^{ij}\delta^{k\ell} +\mu (\delta^{ik}\delta^{j\ell}+\delta^{i\ell}\delta^{jk}),
\end{equation*}
and that $A^{ijk\ell}=A^{jik\ell}=A^{k\ell ij} \in \mathcal{C}^1(\overline{\Omega})$.

We then define the fourth order three-dimensional viscosity tensor in Cartesian coordinates and we denote its components by \gls{B}.
We recall that the contravariant components of this tensor are defined by (see, e.g., \cite{HanSofonea2002})
\begin{equation*}
	B^{ijk\ell}:=\theta \delta^{ij}\delta^{k\ell} +\xi (\delta^{ik}\delta^{j\ell}+\delta^{i\ell}\delta^{jk}),
\end{equation*}
and that $B^{ijk\ell}=B^{jik\ell}=B^{k\ell ij} \in \mathcal{C}^1(\overline{\Omega})$.

For each $\bm{v} \in \bm{H}^1(\Omega)$ we consider the \emph{linearised change of metric tensor in Cartesian coordinates} $\bm{e}(\bm{v})$, whose components $e_{i\|j}(\bm{v})$ are defined by:
\begin{equation*}
\gls{eij}(\bm{v}):=\dfrac{1}{2}(\partial_j v_i +\partial_i v_j) \in L^2(\Omega).
\end{equation*}

This tensor is symmetric, i.e., $e_{i\|j}(\bm{v})=e_{j\|i}(\bm{v})$, for all $\bm{v} \in \bm{H}^1(\Omega)$.
We now state \emph{Korn's inequality} in Cartesian coordinates (see, e.g., Theorem~6.3-4 of~\cite{Ciarlet1988}).
\begin{theorem}
	\label{KornCart}
	Let $\Omega$ be a Lipschitz domain in $\mathbb{R}^3$ and let $\Gamma_0$ be a non-zero area subset of the whole boundary $\Gamma$.
	Then, there exists a constant $\gls{c0}>0$ such that
	$$
	c_0^{-1}\|\bm{v}\|_{\bm{H}^1(\Omega)} \le \|\bm{e}(\bm{v})\|_{\mathbb{L}^2(\Omega)}\le c_0\|\bm{v}\|_{\bm{H}^1(\Omega)},
	$$
	for all $\bm{v} \in \bm{V}(\Omega)$.
	\qed
\end{theorem}

Several proofs have been provided for this intricate inequality. Notably, see \cite{Friedrichs1947}, \cite{Gobert1962}, \cite{HlaNec1970a}, \cite{HlaNec1970b}, page 110 of \cite{DuvLio76}, and Section 6.3 of \cite{NecHla1981}. In \cite{Temam1983}, Korn's inequality is demonstrated in the space $W^{1,p}(\Omega)$ for $1 < p \le \infty$. An elementary proof can be found in \cite{Nit1981}, along with additional details in Appendix (A) of \cite{Miyoshi1985}.

We denote by $\gls{AA}:\bm{V}(\Omega) \to \bm{V}^\ast(\Omega)$ the linear and continuous operator defined in a way that:
\begin{equation}
	\label{opA}
	\langle\mathcal{A}\bm{w},\bm{v}\rangle_{\bm{V}^\ast(\Omega),\bm{V}(\Omega)}:=\int_{\Omega}A^{ijk\ell} e_{k\|\ell}(\bm{w}) e_{i\|j}(\bm{v}) \dd x,\quad\textup{ for all }\bm{w}, \bm{v} \in \bm{V}(\Omega).
\end{equation}

Thanks to the properties of the fourth order three-dimensional elasticity tensor $\{A^{ijk\ell}\}$ it results that the operator $\mathcal{A}$ is symmetric and uniformly positive-definite, in the sense that there exists a constant $\gls{Ce1}>0$ such that:
\begin{equation*}
	\sum_{i,j}|t_{ij}|^2 \le C_e^{(1)} A^{ijk\ell}(x) t_{k\ell}t_{ij},
\end{equation*}
for all $x\in\overline{\Omega}$ and all symmetric matrices $(t_{ij})$.

In a completely similar fashion, we denote by $\gls{BB}:\bm{V}(\Omega) \to \bm{V}^\ast(\Omega)$ the linear and continuous operator defined in a way that:
\begin{equation}
	\label{opB}
	\langle\mathcal{B}\bm{w},\bm{v}\rangle_{\bm{V}^\ast(\Omega),\bm{V}(\Omega)}:=\int_{\Omega}B^{ijk\ell} e_{k\|\ell}(\bm{w}) e_{i\|j}(\bm{v}) \dd x,\quad\textup{ for all }\bm{w}, \bm{v} \in \bm{V}(\Omega).
\end{equation}

Thanks to the properties of the fourth order three-dimensional viscosity tensor $\{B^{ijk\ell}\}$ it results that the operator $\mathcal{B}$ is symmetric and uniformly positive-definite (cf., e.g., \cite{HanSofonea2002}), in the sense that there exists a constant $\gls{Ce2}>0$ such that:
\begin{equation*}
	\sum_{i,j}|t_{ij}|^2 \le C_e^{(2)} B^{ijk\ell}(x) t_{k\ell}t_{ij},
\end{equation*}
for all $x\in\overline{\Omega}$ and all symmetric matrices $(t_{ij})$.

\section{Variational formulation of the time-dependent obstacle problem for a linearly viscoelastic body}
\label{Sec:2}

In this section, we formulate a specific obstacle problem for a linearly viscoelastic body subjected to remaining confined in a prescribed half-space. This condition requires that any admissible displacement vector field $v_i \bm{e}^i$ must ensure that all points of the corresponding deformed configuration remain within a half-space of the form
\begin{equation*}
\gls{H}:=\{x \in \mathbb{R}^3; \bm{Ox} \cdot \bm{q} \ge 0\},
\end{equation*}
where \gls{q} is a \emph{unit-norm vector} that is given once and for all. Let us denote by \gls{I} the identity mapping $\bm{I}:\overline{\Omega} \to \mathbb{E}^3$ and let us assume that the \emph{undeformed} reference configuration satisfies
\begin{equation}
	\label{eq:rest}
	\bm{I}(x) \cdot \bm{q} \ge 0,\quad \textup{ for all }x \in \overline{\Omega},
\end{equation}
or, in other words, the linearly viscoelastic body is located in the prescribed half-space when \emph{at rest and contact may occur. Note that~\eqref{eq:rest} does not imply that the body does not start its motion in contact with the obstacle. The latter property is, in fact, associated with the initial condition.}

The general confinement condition can be thus formulated by requiring that any \emph{admissible displacement vector field} must satisfy
\begin{equation*}
(\bm{I}(x) + v_i(x)\bm{e}^i) \cdot \bm{q} \ge 0,
\end{equation*}
for all $x \in \Gamma$ or, possibly, only for a.a. $x \in \Gamma$ when the covariant components $v_i$ are required to belong to the space $H^1(\Omega)$. The latter is a confinement condition of Signorini type (cf., e.g., \cite{KikuchiOden1988}) and means that the boundary of the linearly viscoelastic body under consideration cannot cross the prescribed half-space identified by its orthonormal complement $\bm{q}$. This condition is widely accepted to imply that all the other points in the deformed reference configuration of the linearly elastic body under consideration must not cross the prescribed half-space.

The subset \gls{UOmega} of admissible displacements is defined by:
\begin{equation*}
	\bm{U}(\Omega):=\{\bm{v}=(v^i) \in \bm{V}(\Omega); (\bm{I} + v_i\bm{e}^i) \cdot \bm{q} \ge 0\, \dd\Gamma\textup{-a.e. on }\Gamma\}.
\end{equation*}

The linearly viscoelastic body under consideration is subjected to applied body forces, that we denote by $\gls{f}=(f_i) \in L^2(0,T;\bm{L}^2(\Omega))$. The applied surface forces for the problem under consideration are expressed in terms of the vectors $\bm{e}^i$ of the Cartesian framework. For the sake of simplicity, we do not consider the action of tractions, even though it is licit to consider the action of traction forces on the portion of the boundary that does not engage contact with the obstacle. Further in the analysis, we will make additional assumptions on the applied body force $\bm{f}$, requiring that $\bm{f}\in H^1(0,T;\bm{L}^2(\Omega))$.

In the problem here discussed, we propose a different concept of solution that is more similar to the classical concept of solutions for variational inequalities~\cite{BrezisStampacchia1968,Glowinski1981,Lions1969}. In our formulation, the velocity of the linearly viscoelastic body needs not vanishing in the long run, as it seems to be postulated in~\cite{CGK24}.

The variational problem $\mathcal{P}(\Omega)$ indicated below constitutes the point of arrival of our analysis. The aim of this paper is to show that one such problem admits a solution enjoying the regularity announced below.

\begin{customprob}{$\mathcal{P}(\Omega)$}\label{problem0}
	Find $\bm{u}=(u_i):(0,T) \to \bm{V}(\Omega)$ such that
	\begin{align*}
		&\bm{u} \in L^\infty(0,T;\bm{U}(\Omega)),\\
		&\dot{\bm{u}} \in \mathcal{C}^0([0,T];\bm{L}^2(\Omega)) \cap L^2(0,T;\bm{V}(\Omega)),
	\end{align*}
	that satisfies the variational inequalities
	\begin{equation*}
		\begin{aligned}
			&2\rho\int_{\Omega}\dot{\bm{u}}(T)\cdot(\bm{v}(T)-\bm{u}(T))\dd x-2\rho\int_{\Omega}\bm{u}_1\cdot(\bm{v}(0)-\bm{u}_0)\dd x-2\rho\int_{0}^{T}\int_{\Omega}\dot{\bm{u}}(t)\cdot(\dot{\bm{v}}(t)-\dot{\bm{u}}(t))\dd x\dd t\\
			&\quad+\int_{0}^{T}\int_{\Omega}A^{ijk\ell} e_{k\|\ell}(\bm{u}(t)) e_{i\|j}(\bm{v}(t)-\bm{u}(t))\dd x\dd t
			+\int_{0}^{T} \int_{\Omega}B^{ijk\ell} e_{k\|\ell}(\dot{\bm{u}}(t)) e_{i\|j}(\bm{v}(t))\dd x\dd t\\
			&\quad-\dfrac{1}{2}\int_{\Omega}B^{ijk\ell} e_{k\|\ell}(\bm{u}(T)) e_{i\|j}(\bm{u}(T)) \dd x+\dfrac{1}{2}\int_{\Omega}B^{ijk\ell} e_{k\|\ell}(\bm{u}_0) e_{i\|j}(\bm{u}_0) \dd x\\
			&\ge\int_{0}^{T} \int_{\Omega} f^i(t) (v_i(t)-u_i(t)) \dd x\dd t,
		\end{aligned}
	\end{equation*}
	for all $\bm{v}=(v_i) \in L^\infty(0,T;\bm{U}(\Omega))$ such that $\dot{\bm{v}}\in L^\infty(0,T;\bm{L}^2(\Omega))\cap L^2(0,T;\bm{V}(\Omega))$, and that satisfies the initial conditions
	\begin{equation*}
		\label{IC}
		\begin{cases}
			\bm{u}(0)=\bm{u}_0,\\
			\dot{\bm{u}}(0)=\bm{u}_1,
		\end{cases}
	\end{equation*}
	where $\gls{u0}=(u_{i,0})\in\bm{U}(\Omega)$, and $\gls{u1}=(u_{i,1}) \in \bm{L}^2(\Omega)$ are prescribed.
	\bqed	
\end{customprob}

Let us note in passing that the initial condition $\bm{u}_0$ is required to belong to the set $\bm{U}(\Omega)$. This means that the linearly viscoelastic body under consideration \emph{could} make contact with the obstacle at the beginning of the observation. This feature could not be considered in~\cite{Pie2020,PWDT3D} due to the intrinsic limitations of the model.

\section{Existence of solutions for Problem~\ref{problem0}}
\label{Sec:3}

The formulation of this initial boundary value problem is inspired by the classical Signorini formulation (cf., e.g., \cite{KikuchiOden1988}).
We construct solutions of Problem~\ref{problem0} by studying its \emph{penalised} version first. In what follows, we denoted by $\gls{kappa}>0$ a real penalty parameter that is meant to approach zero.
In what follows, we denote by $\{f\}^{-}$ the \emph{negative part} of a function $f$, which is defined as $\gls{f-}:=-\min\{f,0\}$.

\begin{customprob}{$\mathcal{P}_\kappa(\Omega)$}\label{problem1}
	Find $\bm{u}_\kappa=(u_{i,\kappa}):(0,T) \to \bm{V}(\Omega)$ such that
	\begin{equation*}
		\begin{aligned}
			&\bm{u}_\kappa \in L^\infty(0,T;\bm{V}(\Omega)),\\
			&\dot{\bm{u}}_\kappa \in L^\infty(0,T;\bm{L}^2(\Omega)) \cap L^2(0,T;\bm{V}(\Omega)),\\
			&\ddot{\bm{u}}_\kappa \in L^\infty(0,T;\bm{V}^\ast(\Omega)),
		\end{aligned}
	\end{equation*}
	that satisfies the variational equations
	\begin{equation*}
		\begin{aligned}
			&2 \rho\, \int_{\Omega}\ddot{\bm{u}}_\kappa(t) \cdot \bm{v} \dd x
			+\int_{\Omega} A^{ijk\ell} e_{k\|\ell}(\bm{u}_\kappa(t)) e_{i\|j}(\bm{v}) \dd x
			+\int_{\Omega} B^{ijk\ell} e_{k\|\ell}(\dot{\bm{u}}_\kappa(t)) e_{i\|j}(\bm{v}) \dd x\\
			&\quad-\dfrac{1}{\kappa}\int_{\Gamma}\{[\bm{I}+u_{i,\kappa}(t)\bm{e}^i]\cdot\bm{q}\}^{-} v_i\bm{e}^i \cdot\bm{q} \dd\Gamma
			= \int_{\Omega} f^{i}(t)v_i \dd x,
		\end{aligned}
	\end{equation*}
	for all $\bm{v}=(v_i) \in \bm{V}(\Omega)$ in the sense of distributions in $(0,T)$, and that satisfies the initial conditions
	\begin{equation*}
		\label{ICkappa}
		\begin{cases}
			\bm{u}_\kappa(0)=\bm{u}_0,\\
			\dot{\bm{u}}_\kappa(0)=\bm{u}_1,
		\end{cases}
	\end{equation*}
	where $\bm{u}_0=(u_{i,0})\in\bm{U}(\Omega)$, and $\bm{u}_1=(u_{i,1}) \in \bm{L}^2(\Omega)$ are the same as in Problem~\ref{problem0}.
	\bqed	
\end{customprob}

Penalised problems akin to Problem~\ref{problem1} have been addressed in the literature by several authors; see for instance~\cite{BJ10,BJ11,BockJarSil2016}. In~\cite{BockJarSil2016}, the solution strategy exploits the underlying regularity of the model as well as the fact that  the model is posed over a two-dimensional Lipschitz domain. In~\cite{BJ10,BJ11}, the argument hinges on the fact that the competitors to the role of solutions are chosen among those whose weak derivative in time exists and is of class $L^2$. The competitors to the role of solutions for Problem~\ref{problem0} are instead only of class $\mathcal{C}^0$ in time.

In the case treated in this paper, for a.a. $t\in(0,T)$, the displacement solving Problem~\ref{problem0} is, in general, of class $\bm{H}^1(\Omega)$, where $\Omega$ is a Lipschitz domain in $\mathbb{R}^3$.
Let us recall the following general result which can be found in, e.g., Lemma~3.1 in~\cite{PieTem2023} or \cite{ABM14,EG15}.
\begin{lemma}
	\label{lem:1}
	Let $\mathscr{O} \subset \mathbb{R}^m$, with $m\ge 1$ an integer, be a non-empty open set. The operator $-\{\cdot\}^{-}:L^2(\mathscr{O}) \to L^2(\mathscr{O})$ defined by
	$$
	L^2(\mathscr{O})\ni f \mapsto -\{f\}^{-}:=\min\{f,0\} \in L^2(\mathscr{O}),
	$$
	is monotone, bounded and Lipschitz continuous with Lipschitz constant equal to $1$.
	\qed
\end{lemma}

We re-cast the penalty term, which describes the energy that has to be \emph{paid} by the system to violate the constraint, in terms of the non-linear operator $\gls{N}:\bm{V}(\Omega) \to \bm{V}^\ast(\Omega)$, that we define by $\mathcal{N}:=\left(\textup{tr}^\star\circ(-\{\cdot\}^{-})\circ\textup{tr}\right)$, where \gls{trstar} is the Banach adjoint of $\textup{tr}$. The operator $\mathcal{N}$ is such that
\begin{equation*}
	\langle \mathcal{N}\bm{w},\bm{v}\rangle_{\bm{V}^\ast(\Omega),\bm{V}(\Omega)}:=-\int_{\Gamma} \{[\bm{I}+w_i\bm{e}^i]\cdot\bm{q}\}^{-} v_i\bm{e}^i \cdot\bm{q} \dd\Gamma,
\end{equation*}
for all $\bm{w}=(w_i) \in \bm{V}(\Omega)$ and all $\bm{v}=(v_i)\in \bm{V}(\Omega)$. Observe that the latter definition makes sense being the trace of an element in $\bm{V}(\Omega)$ of class $\bm{L}^2(\Gamma)$ when $\Omega$ is a Lipschitz domain as in our case.
As a result, it is straightforward to observe that also the operator $\mathcal{N}$ is monotone, bounded and Lipschitz continuous.

We then define the non-linear operator $\gls{tildeN}:L^2(0,T;\bm{V}(\Omega)) \to L^2(0,T;\bm{V}^\ast(\Omega))$ pointwise by
\begin{equation}
	\label{Ntilde}
	(\tilde{\mathcal{N}}\bm{w})(t):={\mathcal{N}}(\bm{w}(t)),\quad\textup{ for a.a. } t \in (0,T),
\end{equation}
and we observe that, in light of the properties of the non-linear operator $\mathcal{N}$ introduced beforehand, the operator $\tilde{\mathcal{N}}$ is monotone, bounded and Lipschitz continuous.

We now recall the very important Gronwall's inequality often used to study evolutionary problems (see, e.g.  the original paper~\cite{GW} or, for instance, Appendix~B.2(j) in~\cite{Evans2010} or Theorem~1.1 in Chapter~III of~\cite{Har02}).
\begin{theorem}
	\label{Gronwall}
	Let $T>0$ and suppose that the function $y:[0,T]\to \mathbb R$ is absolutely continuous and such that
	$$
	\dot{y}(t) \le a(t) y(t)+b(t), \quad\textup{ for a.a. } t\in (0,T),
	$$
	where $a, b \in L^1(0,T)$ and $a(t), b(t) \ge 0$ for a.a. $t\in (0,T)$.
	Then, it results:
	$$
	y(t) \le \left[y(0)+\int_0^t b(s) \dd s\right]\exp\left(\int_0^t a(s) \dd s\right), \quad\textup{ for all } t\in [0,T].
	$$
	\qed
\end{theorem}

Let us also recall a classical result in the analysis of evolutionary equations: the Aubin-Lions-Simon theorem (cf., e.g., Theorem~8.62 in~\cite{Leoni2017}).

\begin{theorem}[Aubin-Lions-Simon]
	\label{th:ALS}
	Let $I\subset\mathbb{R}$ be an open and bounded interval. Let $(Y_0,\|\cdot\|_{Y_0})$, $(Y_1,\|\cdot\|_{Y_1})$ and $(Y_2,\|\cdot\|_{Y_2})$ be Banach spaces such that:
	\begin{equation*}
		Y_0\hookrightarrow\hookrightarrow Y_1 \hookrightarrow Y_2.
	\end{equation*}
	
	Let $1\le p <\infty$ and let $1\le q\le \infty$, and let $\mathcal{V}$ be the Banach space defined by $\mathcal{V}:=\{u\in L^p(I;Y_0);\dot{u}\in L^q(I;Y_2)\}$, and equipped with the norm:
	\begin{equation*}
		\|u\|_{\mathcal{V}}:=\|u\|_{L^p(I;Y_0)}+\|\dot{u}\|_{L^q(I;Y_2)},\quad\textup{ for all }u\in \mathcal{V}.
	\end{equation*}
	
	Then, the continuous embedding $\mathcal{V}\hookrightarrow L^p(I;Y_1)$ is compact. If $p=\infty$ and $q>1$ then the continuous embedding $\mathcal{V}\hookrightarrow \mathcal{C}^0(\overline{I};Y_1)$ is compact.
	\qed
\end{theorem}

The next step consists in showing, by Galerkin method, that Problem~\ref{problem1} admits at least one solution.
Prior to establishing this result, we need a lemma on the compactness of the time-dependent version of the trace operator.
We establish this lemma in a general setting.

\begin{lemma}
	\label{lem:tr}
	Let $\bm{\mathcal{V}}:=\{\bm{v} \in L^2(0,T;\bm{V}(\Omega)); \dot{\bm{v}} \in L^2(0,T;\bm{V}^\ast(\Omega))\}$.
	Let $\textup{tr}:\bm{V}(\Omega)\to\bm{L}^2(\Gamma)$ denote the classical trace operator.
	Then the operator $\tilde{\textup{tr}}:\bm{\mathcal{V}} \to L^2(0,T;\bm{L}^2(\Gamma))$, that constitutes the time-dependent version of $\textup{tr}$, and that is defined pointwise by
	\begin{equation*}
		(\tilde{\textup{tr}}\,\bm{v})(t):=\textup{tr}(\bm{v}(t)),\quad\textup{ for a.a. }t\in (0,T),
	\end{equation*}
	at each $\bm{v}\in \bm{\mathcal{V}}$ is linear, continuous and compact.
\end{lemma}
\begin{proof}
	The linearity and continuity properties are straightforward, and the proof closely follows the strategy for establishing the linearity and continuity of the time-dependent version of the linearised change of metric tensor (cf., e.g., page~4 in~\cite{Pie2019}).
	
	To establish the compactness, consider a sequence $\{\bm{v}_n\}_{n=1}^\infty$ that is bounded in $\bm{\mathcal{V}}$.  An application of the Aubin-Lions-Simon theorem (Theorem~\ref{th:ALS}) shows that $\bm{v}_n \to \bm{v}$ in $L^2(0,T;\bm{L}^2(\Omega))$ as $n\to\infty$ up to passing to a subsequence.
	Thanks to Theorem~18.1(iii) and Corollary~18.4 in~\cite{Leoni2017}, there exits two constants $c=c(\Omega)>0$ and $\epsilon=\epsilon(\Omega)>0$ which are independent of $n$ and $t$ such that:
	\begin{equation}
		\label{est:1}
		\int_{\Gamma} |\textup{tr}(\bm{v}_n(t))-\textup{tr}(\bm{v}(t))|^2 \dd\Gamma \le \dfrac{c}{\varepsilon}\int_{\Omega} |\bm{v}_n(t)-\bm{v}(t)|^2 \dd x
		+c\varepsilon\sum_{i=1}^3\int_{\Omega} |\nabla v_{i,n}(t)-\nabla v_i(t)|^2 \dd x,
	\end{equation}
	for all $0<\varepsilon<\epsilon$ and for a.a. $t\in (0,T)$. Integrating~\eqref{est:1} in $(0,T)$ gives:
	\begin{equation}
		\label{est:2}
		\begin{aligned}
			&\int_{0}^{T}\int_{\Gamma} |(\tilde{\textup{tr}}\,\bm{v}_n)(t)-(\tilde{\textup{tr}}\,\bm{v})(t)|^2 \dd\Gamma \dd t=\int_{0}^{T}\int_{\Gamma} |\textup{tr}(\bm{v}_n(t))-\textup{tr}(\bm{v}(t))|^2 \dd\Gamma \dd t\\
			&\le \dfrac{c}{\varepsilon}\int_{0}^{T}\int_{\Omega} |\bm{v}_n(t)-\bm{v}(t)|^2 \dd x \dd t+c\varepsilon\sum_{i=1}^3\int_{0}^{T}\int_{\Omega} |\nabla v_{i,n}(t)-\nabla v_i(t)|^2 \dd x \dd t.
		\end{aligned}
	\end{equation}
	
	Consequently, the first term on the right-hand side of the inequality in~\eqref{est:2} tends to zero as $n\to\infty$. 
	
	Additionally, thanks to the assumed boundedness for the sequence $\{\bm{v}_n\}_{n=1}^\infty$ in $L^2(0,T;\bm{V}(\Omega))$, the second term on the right-hand side of the inequality in~\eqref{est:2} is bounded independently of $n$.
	Therefore, there exists a constant $\tilde{C}>0$ independent of $t$ and $n$ such that:
	\begin{equation}
		\label{est:3}
		\limsup_{n\to\infty}\int_{0}^{T}\int_{\Gamma} |(\tilde{\textup{tr}}\,\bm{v}_n)(t)-(\tilde{\textup{tr}}\bm{v})(t)|^2 \dd\Gamma \dd t \le \tilde{C}\varepsilon.
	\end{equation}
	
	Thanks to the arbitrariness of $0<\varepsilon\le \epsilon$, we obtain that:
	\begin{equation*}
		\lim_{n\to\infty}\int_{0}^{T}\int_{\Gamma} |(\tilde{\textup{tr}}\,\bm{v}_n)(t)-(\tilde{\textup{tr}}\bm{v})(t)|^2 \dd\Gamma \dd t=0,
	\end{equation*}
	thus showing that $\tilde{\textup{tr}}\,\bm{v}_n \to \tilde{\textup{tr}}\,\bm{v}$ in $L^2(0,T;\bm{L}^2(\Gamma))$ as $n\to\infty$, and establishing the sought compactness. Note that the subsequence for which the pre-compactness of $\{\tilde{\textup{tr}}\,\bm{v}_n\}_{n=1}^\infty$ is realised is exactly the one for which the convergence $\bm{v}_n \to \bm{v}$ in $L^2(0,T;\bm{L}^2(\Omega))$ as $n\to\infty$ asserted by the Aubin-Lions-Simon theorem holds.
\end{proof}

Next, we establish the continuity of the operators $\mathcal{A}$ and $\mathcal{B}$.

\begin{lemma}
	\label{lem:2}
	The operator $\mathcal{A}:\bm{V}(\Omega) \to \bm{V}^\ast(\Omega)$ defined in~\eqref{opA} is linear and continuous. Moreover, also the corresponding time-dependent version $\gls{tildeA}:L^\infty(0,T;\bm{V}(\Omega)) \to L^\infty(0,T;\bm{V}^\ast(\Omega))$ defined pointwise by
	\begin{equation*}
		(\tilde{\mathcal{A}}\bm{w})(t):=\mathcal{A}(\bm{w}(t)),\quad\textup{ for all } \bm{w}\in L^\infty(0,T;\bm{V}(\Omega)),
	\end{equation*}
	for a.a. $t\in(0,T)$ is linear and continuous.
\end{lemma}
\begin{proof}
	The operator $\mathcal{A}$ is the classical operator in linearised elasticity; its linearity and boundedness are straightforward to establish (cf., e.g., \cite{Ciarlet1988}).
	For what concerns the time-dependent version of this operator, the linearity and boundedness follow, in the same spirit as in \cite{Pie2019}, from the uniform boundedness of the components $A^{ijk\ell}$ of the fourth order three-dimensional elasticity tensor.
\end{proof}

In a similar fashion, the following preparatory result can be established.

\begin{lemma}
	\label{lem:3}
	The operator $\mathcal{B}:\bm{V}(\Omega) \to \bm{V}^\ast(\Omega)$ defined in~\eqref{opB} is linear and continuous. Moreover, also the corresponding time-dependent version $\gls{tildeB}:L^2(0,T;\bm{V}(\Omega)) \to L^2(0,T;\bm{V}^\ast(\Omega))$ defined pointwise by
	\begin{equation*}
		(\tilde{\mathcal{B}}\bm{w})(t):=\mathcal{B}(\bm{w}(t)),\quad\textup{ for all } \bm{w}\in L^2(0,T;\bm{V}(\Omega)),
	\end{equation*}
	for a.a. $t\in(0,T)$ is linear and continuous.
	\qed
\end{lemma}

We are now ready to establish the existence of solutions for Problem~\ref{problem1} via Galerkin method. We note in passing that the presence of the non-linear term associated with the extent to which the constraint is broken renders the problem challenging. The proof of the next result hinges on the compactness of the time-dependent version of the trace operator established in Lemma~\ref{lem:tr}.

\begin{theorem}
	\label{th:1}
	Assume that $\bm{f}=(f^i) \in L^2(0,T;\bm{L}^2(\Omega))$.
	Then Problem~\ref{problem1} admits at least one solution.
\end{theorem}
\begin{proof}
	For the sake of clarity, we break the proof into three parts, numbered (i)--(iii).
	
	(i) \emph{Construction of a Galerkin approximation}.
	Since $\bm{V}(\Omega)$ is an infinite dimensional separable Hilbert space which is dense in $\bm{L}^2(\Omega)$ and compactly embedded in $\bm{L}^2(\Omega)$ in light of the Rellich-Kondra\v{s}ov theorem (cf., e.g., Theorem~8.4-3 in~\cite{Cia25}), we infer that there exists an orthogonal basis $\{\bm{w}^p\}_{p=1}^\infty$ of the space $\bm{V}(\Omega)$, whose elements also constitute a Hilbert basis of the space $\bm{L}^2(\Omega)$.
	
	The existence of such a basis is assured by the spectral theorem (Theorem~6.2-1 of~\cite{RT83}).
	For each positive integer $m\ge 1$, we denote by $\bm{E}^m$ the following $m$-dimensional linear hull:
	\begin{equation*}
		\bm{E}^m:=\textup{Span }\{\bm{w}^p\}_{p=1}^m \subset \bm{V}(\Omega) \subset \bm{L}^2(\Omega).
	\end{equation*}
	
	For each integer $1\le p\le m$, let $\lambda_p>0$ be defined in a way such that:
	\begin{equation*}
		\int_{\Omega}A^{ijk\ell}e_{k\|\ell}(\bm{w}^p)e_{i\|j}(\bm{v})\dd x=\lambda_p\int_{\Omega}(\bm{w}^p\cdot\bm{v})\dd x,\quad\textup{ for all }\bm{v}\in\bm{V}(\Omega).
	\end{equation*}
	
	Since each element of this Hilbert basis is independent of the variable $t$, we have that $\bm{w}^p \in L^\infty(0,T;\bm{V}(\Omega))$ for each integer $1 \le p \le m$. We now discretise Problem~\ref{problem1}.
	
	\begin{customprob}{$\mathcal{P}_\kappa^{m}(\Omega)$}\label{problem2}
		Find functions $c_{p,\kappa}:[0,T] \to \mathbb R$, $1 \le p \le m$, such that
		\begin{equation*}
			\bm{u}^m_\kappa(t):=\sum_{p=1}^{m} c_{p,\kappa}(t) \bm{w}^p, \quad \textup{ for a.a. } t \in (0,T),
		\end{equation*}
		and satisfying the variational equations
		\begin{equation*}
			\begin{aligned}
				&2 \rho \int_{\Omega}\ddot{\bm{u}}_\kappa^m(t) \cdot \bm{w}^p \dd x
				+\int_{\Omega} A^{ijk\ell} e_{k\|\ell}(\bm{u}_\kappa^m(t)) e_{i\|j}(\bm{w}^p) \dd x
				+\int_{\Omega} B^{ijk\ell} e_{k\|\ell}(\dot{\bm{u}}_\kappa^m(t)) e_{i\|j}(\bm{w}^p) \dd x\\
				&\quad-\dfrac{1}{\kappa}\int_{\Gamma}\{[\bm{I}+u_{i,\kappa}^m(t)\bm{e}^i]\cdot\bm{q}\}^{-} w_i^p\bm{e}^i \cdot\bm{q} \dd\Gamma
				= \int_{\Omega} f^{i}(t)w_i^p \dd x,
			\end{aligned}
		\end{equation*}
		for all $1 \le p \le m$ in the sense of distributions in $(0,T)$, and that satisfies the initial conditions
		\begin{equation*}
			\label{ICkappam}
			\begin{cases}
				\bm{u}_\kappa^{m}(0)=\bm{u}_0^m,\\
				\dot{\bm{u}}_\kappa^{m}(0)=\bm{u}_1^m,
			\end{cases}
		\end{equation*}
		where the initial conditions $\bm{u}_0^m$ and $\bm{u}_1^m$ are defined by:
		\begin{equation*}
			\begin{aligned}
				\bm{u}_0^m:&=\sum_{p=1}^m \left(\int_{\Omega}A^{ijk\ell}e_{k\|\ell}(\bm{u}_0)e_{i\|j}(\bm{w}^p)\dd x\right)\dfrac{\bm{w}^p}{\lambda_p},\\
				\bm{u}_1^m:&=\sum_{p=1}^m \int_{\Omega}(\bm{u}_1\cdot\bm{w}^p) \dd x\bm{w}^p.
			\end{aligned}
		\end{equation*}
		\bqed	
	\end{customprob}
	
	We observe (cf., e.g., Theorem~4.13-1 in~\cite{Cia25}) that $\bm{u}_0^m \to \bm{u}_0$ in $\bm{V}(\Omega)$ as $m\to\infty$, and that $\bm{u}_1^m \to \bm{u}_1$ in $\bm{L}^2(\Omega)$ as $m\to\infty$.
	Since the coefficients $c_{p,\kappa}$ and their derivatives only depend on the time variable, we can take them outside of the integral sign, getting a  $m \times m$ non-linear system of second order ordinary differential equations with respect to the variable $t$. Such a system can be rewritten in the form
	\begin{equation}
		\label{ODE}
		\begin{aligned}
			2\rho\ddot{\bm{C}}_\kappa(t)&=\left(-\int_{\Omega}A^{ijk\ell} e_{k\|\ell}(\bm{w}^r) e_{i\|j}(\bm{w}^p)\dd x\right)_{p,r=1}^m \bm{C}_\kappa(t)\\
			&\quad+\left(-\int_{\Omega}B^{ijk\ell} e_{k\|\ell}(\bm{w}^r) e_{i\|j}(\bm{w}^p)\dd x\right)_{p,r=1}^m \dot{\bm{C}}_\kappa(t)\\
			&\quad+ \dfrac{1}{\kappa} \left(\int_{\Gamma} \left(\{[\bm{I} +(\bm{C}_\kappa(t) \cdot (w^1_i\dots w^m_i))
			\bm{e}^{i}]\cdot \bm{q}\}^-\right) (w^p_{i} \bm{e}^{i} \cdot\bm{q}) \dd x\right)_{p=1}^m\\
			&\quad+\left(\int_{\Omega}f^i(t) w_i^p \dd x\right)_{p=1}^m,
		\end{aligned}
	\end{equation}
	where $\bm{C}_\kappa(t):=(c_{1,\kappa}(t) \dots c_{m,\kappa}(t))$.
	Thanks to Lemma~\ref{lem:1}, the right hand side of~\eqref{ODE} is Lipschitz continuous in $\mathbb{R}^m$ uniformly with respect to $t$, since it does not explicitly depend on $t$.
	An application of Theorem~1.45 of~\cite{Rou13} gives that for each integer $m \ge 1$ there exists a unique global solution $\bm{u}_\kappa^m$ to Problem~\ref{problem2}, defined a.e. over the interval $(0,T)$, such that:
	\begin{equation*}
		\begin{aligned}
			&\bm{u}^m_\kappa \in L^\infty(0,T;\bm{E}^m),\\
			&\dot{\bm{u}}^m_\kappa \in L^\infty(0,T;\bm{E}^m),\\
			&\ddot{\bm{u}}^m_\kappa \in L^2(0,T;\bm{E}^m).
		\end{aligned}
	\end{equation*}
	
	(ii) \emph{Energy estimates for the approximate solutions}. Let us multiply the variational equations in Problem~\ref{problem2} by $\dot{c}_{p,\kappa}(t)$, with $0<t<T$, and sum with respect to $p\in\{1, \dots,m\}$. The penalised variational equations in Problem~\ref{problem2} take the form
	\begin{equation}
		\label{pm1}
		\begin{aligned}
			&\rho \dfrac{\dd}{\dd t} \int_{\Omega}\dot{u}_{i,\kappa}^m(t) \dot{u}_{i,\kappa}^m(t) \dd x+\dfrac{1}{2} \dfrac{\dd}{\dd t} \int_{\Omega} A^{ijk\ell} e_{k\|\ell}(\bm{u}^m_\kappa(t))  e_{i\|j}(\bm{u}^m_\kappa(t)) \dd x
			+\int_{\Omega} B^{ijk\ell} e_{k\|\ell}(\dot{\bm{u}}^m_\kappa(t)) e_{i\|j}(\dot{\bm{u}}^m_\kappa(t)) \dd x\\
			&\quad+\dfrac{1}{2\kappa}\dfrac{\dd}{\dd t}\left(\int_{\Gamma} \left(\left\{[\bm{I}+u^m_{i,\kappa}(t) \bm{e}^i]\cdot \bm{q}\right\}^{-}\right)^2 \dd\Gamma\right)
			=\int_{\Omega} f^{i}(t) \dot{u}^{m}_{i,\kappa}(t) \dd x,
		\end{aligned}
	\end{equation}
	and are valid in the sense of distributions in $(0,T)$. 
	Observe that the differentiation of the negative part is obtained as a result of the same computational steps as in Stampacchia's theorem (cf., e.g., Theorem~4.4 on page~153 of~\cite{EG15}), together with an application of Theorem~8.28 of~\cite{Leoni2017}.
	An integration over the interval $(0,t)$, where $0<t\le T$, changes \eqref{pm1} into:
	\begin{equation}
		\label{pm1-2}
		\begin{aligned}
			&\rho \int_{\Omega} \dot{u}_{i,\kappa}^m(t) \dot{u}_{i,\kappa}^m(t) \dd x
			+\dfrac{1}{2} \int_{\Omega} A^{ijk\ell} e_{k\|\ell}(\bm{u}^m_\kappa(t)) e_{i\|j}(\bm{u}^m_\kappa(t))\dd x\\
			&\quad+\int_{0}^{t} \int_{\Omega} B^{ijk\ell} e_{k\|\ell}(\dot{\bm{u}}^m_\kappa(\tau)) e_{i\|j}(\dot{\bm{u}}^m_\kappa(\tau))\dd x \dd\tau+\dfrac{1}{2\kappa} \int_{\Gamma} \left(\left\{[\bm{I}+u^{m}_{i,\kappa}(t) \bm{e}^{i}]\cdot \bm{q}\right\}^{-}\right)^2 \dd\Gamma\\
			&=\rho \int_{\Omega}u_{i,1}^m u_{i,1}^m \dd x
			+\dfrac{1}{2}\int_{\Omega} A^{ijk\ell} e_{k\|\ell}(\bm{u}_{0}^m) e_{i\|j}(\bm{u}_{0}^m) \dd x
			+\dfrac{1}{2\kappa} \int_{\Gamma} \left(\left\{[\bm{I}+u^{m}_{i,0} \bm{e}^{i}]\cdot \bm{q}\right\}^{-}\right)^2 \dd\Gamma\\
			&\quad+\int_{0}^{t}\int_{\Omega} f^{i}(\tau) \dot{u}^m_{i,\kappa}(\tau) \dd x \dd\tau.
		\end{aligned}
	\end{equation}
	
	Let us observe that the third integral in the right-hand side of~\eqref{pm1-2} tends to zero, since $\bm{u}_0^m \to \bm{u}_0$ in $\bm{V}(\Omega)$ as $m\to\infty$, and $\bm{u}_0 \in \bm{U}(\Omega)$ by assumption.
	An application of the Cauchy-Schwarz inequality gives:
	\begin{equation}
		\label{pm1-2-3}
		\begin{aligned}
			&\int_{0}^{t}\int_{\Omega} f^{i}(\tau) \dot{u}^m_{i,\kappa}(\tau) \dd x \dd\tau
			\le\left(\int_0^{T} \|\bm{f}(t)\|_{\bm{L}^2(\Omega)}^2 \dd t\right)^{1/2} \left(\int_{0}^{t}\|\dot{\bm{u}}^m_\kappa(\tau)\|_{\bm{L}^2(\Omega)}^2\dd\tau\right)^{1/2}\\ 
			&\le \dfrac{c_0^2 C_e^{(2)}}{2}\int_0^{T} \|\bm{f}(t)\|_{\bm{L}^2(\Omega)}^2 \dd t +\dfrac{1}{2c_0^2C_e^{(2)}}\int_{0}^{t}\|\dot{\bm{u}}^m_\kappa(\tau)\|_{\bm{L}^2(\Omega)}^2 \dd\tau.
		\end{aligned}
	\end{equation}
	
	By the uniform positive-definiteness of the elasticity tensors $\{A^{ijk\ell}\}_{i,j,k,\ell}$ and $\{B^{ijk\ell}\}_{i,j,k,\ell}$, Korn's inequality (Theorem~\ref{KornCart}), \eqref{pm1-2}, and~\eqref{pm1-2-3}, the following estimate holds for $m$ sufficiently large:
	\begin{equation}
		\label{pm2}
		\begin{aligned}
			&\rho\|\dot{\bm{u}}^m_\kappa(t)\|_{\bm{L}^2(\Omega)}^2 + \dfrac{\|\bm{u}^m_\kappa(t)\|_{\bm{V}(\Omega)}^2}{2c_0^2 C_e^{(1)}}+\dfrac{1}{2c_0^2 C_e^{(2)}}\int_{0}^{t}\|\dot{\bm{u}}^m_\kappa(\tau)\|_{\bm{V}(\Omega)}^2 \dd \tau+\dfrac{1}{2\kappa}\left\|\left\{[\bm{I}+u^m_{i,\kappa}(t) \bm{e}^{i}]\cdot \bm{q}\right\}^{-}\right\|_{L^2(\Gamma)}^2\\ 
			&\le \rho\|\bm{u}_{1}^m\|_{\bm{L}^2(\Omega)}^2 +\dfrac{c_0^2 C_e^{(2)}}{2}\|\bm{f}\|_{L^2(0,T;\bm{L}^2(\Omega))}^2+\dfrac{1}{2}\int_{\Omega}A^{ijk\ell} e_{k\|\ell}(\bm{u}_{0}^m) e_{i\|j}(\bm{u}_{0}^m) \dd x\\
			&\quad+\rho\int_{0}^{t}\|\dot{\bm{u}}^m_\kappa(\tau)\|_{\bm{L}^2(\Omega)}^2 \dd\tau + \dfrac{1}{2c_0^2 C_e^{(1)}}\int_{0}^{t}\|\bm{u}^m_\kappa(\tau)\|_{\bm{V}(\Omega)}^2\dd\tau
			+\dfrac{1}{2\kappa}\int_{0}^{t}\left\|\left\{[\bm{I}+u^{m}_{i,\kappa}(\tau) \bm{e}^{i}]\cdot \bm{q}\right\}^{-}\right\|_{L^2(\Gamma)}^2 \dd\tau.
		\end{aligned}
	\end{equation}
	
	An application of the Gronwall's inequality (Theorem~\ref{Gronwall}) with 
	\begin{equation*}
		y(t):=\rho\int_{0}^{t}\|\dot{\bm{u}}^m_\kappa(\tau)\|_{\bm{L}^2(\Omega)}^2 \dd\tau + \dfrac{1}{2c_0^2 C_e^{(1)}}\int_{0}^{t}\|\bm{u}^m_\kappa(\tau)\|_{\bm{V}(\Omega)}^2\dd\tau+\dfrac{1}{2\kappa}\int_{0}^{t} \left\|\left\{[\bm{I}+u^{m}_{i,\kappa}(\tau) \bm{e}^{i}]\cdot \bm{q}\right\}^{-}\right\|_{L^2(\Gamma)}^2 \dd\tau,
	\end{equation*}
	$a \equiv 1 >0$ and
	$$
	b \equiv\max\left\{\dfrac{1}{2},\rho,\dfrac{c_0^2 C_e^{(2)}}{2}\right\}\left(\|\bm{u}_{1}^m\|_{\bm{L}^2(\Omega)}^2 +\|\bm{f}\|_{L^2(0,T;\bm{L}^2(\Omega))}^2+\int_{\Omega}A^{ijk\ell} e_{k\|\ell}(\bm{u}_{0}^m) e_{i\|j}(\bm{u}_{0}^m) \dd x\right) \ge 0,
	$$
	gives the following upper bound
	\begin{equation}
		\label{pm2-2}
		\begin{aligned}
			&\rho\int_{0}^{t}\|\dot{\bm{u}}^m_\kappa(\tau)\|_{\bm{L}^2(\Omega)}^2 \dd\tau + \dfrac{1}{2c_0 C_e^{(1)}}\int_{0}^{t}\|\bm{u}^m_\kappa(\tau)\|_{\bm{V}(\Omega)}^2\dd\tau
			+\dfrac{1}{2\kappa}\int_{0}^{t}\left\|\left\{[\bm{I}+u^{m}_{i,\kappa}(\tau) \bm{e}^{i}]\cdot \bm{q}\right\}^{-}\right\|_{L^2(\Gamma)}^2 \dd\tau\\
			&\le t \max\left\{\dfrac{1}{2},\rho,\dfrac{c_0 C_e^{(2)}}{2}\right\}\left(\|\bm{u}_{1}^m\|_{\bm{L}^2(\Omega)}^2 +\|\bm{f}\|_{L^2(0,T;\bm{L}^2(\Omega))}^2+\int_{\Omega}A^{ijk\ell} e_{k\|\ell}(\bm{u}_{0}^m) e_{i\|j}(\bm{u}_{0}^m) \dd x\right)e^{t},
		\end{aligned}
	\end{equation}
	which can be easily made independent of $t$. Observe that the previous right hand side is bounded independently of $m$ in light of Theorem~4.13-1 in~\cite{Cia25}. Therefore, combining~\eqref{pm2} with~\eqref{pm2-2} gives
	\begin{equation}
		\label{pm5}
		\begin{aligned}
			\{\bm{u}^m_\kappa\}_{m=1}^\infty &\textup{ is bounded in }L^\infty(0,T;\bm{V}(\Omega)) \textup{ independently of }m \textup{ and }\kappa,\\
			\{\dot{\bm{u}}^m_\kappa\}_{m=1}^\infty &\textup{ is bounded in }L^\infty(0,T;\bm{L}^2(\Omega)) \cap L^2(0,T;\bm{V}(\Omega)) \textup{ independently of }m \textup{ and }\kappa,
		\end{aligned}
	\end{equation}
	and, moreover, by~\eqref{pm2-2} and Gronwall's inequality (cf., e.g., \cite{GW} or Appendix~B.2(j) in~\cite{Evans2010}), it results:
	\begin{equation}
		\label{pm6}
		\left\|\left\{[\bm{I}+u^{m}_{i,\kappa} \bm{e}^{i}]\cdot \bm{q}\right\}^{-}\right\|_{L^2(0,T;L^2(\Gamma))} \le \sqrt{\kappa T b e^T}.
	\end{equation}	
	
	Since the following direct sum decomposition holds
	$$
	\bm{V}(\Omega)=\bm{E}^m \oplus (\bm{E}^m)^\perp,
	$$
	we get that for any $\bm{v} \in \bm{V}(\Omega)$ with $\|\bm{v}\|_{\bm{V}(\Omega)} = 1$, and a.a. $t \in (0,T)$, the variational equations in Problem~\ref{problem2} give
	\begin{align*}
		&2\rho|\langle\ddot{\bm{u}}^m_\kappa(t) , \bm{v} \rangle_{\bm{V}^\ast(\Omega),\bm{V}(\Omega)}|\le \|\bm{f}(t)\|_{\bm{L}^2(\Omega)}+ \left(\max_{i,j,k,\ell}\|A^{ijk\ell}\|_{\mathcal{C}^0(\overline{\Omega})}\right)\|\bm{u}^m_\kappa(t)\|_{\bm{V}(\Omega)}\\ &\quad+\left(\max_{i,j,k,\ell}\|B^{ijk\ell}\|_{\mathcal{C}^0(\overline{\Omega})}\right)\|\dot{\bm{u}}^m_\kappa(t)\|_{\bm{V}(\Omega)}
		+\dfrac{1}{\kappa}\left\|\left\{[\bm{I}+u^{m}_{i,\kappa}(t) \bm{e}^{i}]\cdot \bm{q}\right\}^{-}\bm{q}\right\|_{\bm{L}^2(\Gamma)},
	\end{align*}
	and, by~\eqref{pm5} and~\eqref{pm6}, we thus infer that there exists a constant $C>0$, independent of $m$, $t$ and $\kappa$, such that:
	\begin{equation}
		\label{pm7}
		\|\ddot{\bm{u}}^m_\kappa\|_{L^2(0,T;\bm{V}^\ast(\Omega))} \le \dfrac{C}{2\rho}\left(1+\dfrac{1}{\sqrt{\kappa}}\right).
	\end{equation}
	
	(iii) \emph{Passage to the limit as $m\to\infty$ and completion of the proof}. By~\eqref{pm5}, \eqref{pm6}, and~\eqref{pm7} we can infer that there exist subsequences, still denoted $\{\bm{u}^m\}_{m=1}^\infty$, $\{\dot{\bm{u}}^m_\kappa\}_{m=1}^\infty$ and $\{\ddot{\bm{u}}^m_\kappa\}_{m=1}^\infty$ such that the following convergences hold:
	\begin{equation}
		\label{c1}
		\begin{aligned}
			\bm{u}^m_\kappa \wsc \bm{u}_\kappa,\quad&\textup{ in }L^\infty(0,T;\bm{V}(\Omega)) \textup{ as }m \to\infty,\\
			\dot{\bm{u}}^m_\kappa \wsc \dot{\bm{u}}_\kappa,\quad&\textup{ in }L^\infty(0,T;\bm{L}^2(\Omega)) \textup{ as }m \to\infty,\\
			\dot{\bm{u}}^m_\kappa \rightharpoonup \dot{\bm{u}}_\kappa,\quad&\textup{ in }L^2(0,T;\bm{V}(\Omega)) \textup{ as }m \to\infty,\\
			\ddot{\bm{u}}^m_\kappa\rightharpoonup\ddot{\bm{u}}_\kappa,\quad&\textup{ in }L^2(0,T;\bm{V}^\ast(\Omega)) \textup{ as }m \to\infty,\\
			\kappa^{-1}\left\{[\bm{I}+u^{m}_{i,\kappa} \bm{e}^{i}]\cdot \bm{q}\right\}^{-} \rightharpoonup \chi_\kappa ,\quad&\textup{ in }L^2(0,T;L^2(\Gamma)) \textup{ as }m \to\infty.
		\end{aligned}
	\end{equation}
	
	By the Sobolev embedding theorem (cf., e.g., Theorem~10.1.20 of~\cite{YP}), we obtain
	\begin{equation}
		\label{unif}
		\begin{aligned}
			&\bm{u}^m_\kappa \rightharpoonup \bm{u}_\kappa, \quad\textup{ in }\mathcal{C}^0([0,T];\bm{V}(\Omega)) \textup{ as } m \to \infty,\\
			&\dot{\bm{u}}^m_\kappa \rightharpoonup \dot{\bm{u}}_\kappa, \quad\textup{ in }\mathcal{C}^0([0,T];\bm{V}^\ast(\Omega)) \textup{ as } m \to \infty,
		\end{aligned}
	\end{equation}
	
	An application of Theorem~8.28 of~\cite{Leoni2017} to the fifth convergence of the process~\eqref{c1} gives:
	\begin{equation}
		\label{penaltym}
		\kappa^{-1}\left\{[\bm{I}+u^{m}_{i,\kappa} \bm{e}^{i}]\cdot \bm{q}\right\}^{-} \rightharpoonup \chi_\kappa ,\quad\textup{ in }L^2((0,T)\times\Gamma) \textup{ as }m \to\infty.
	\end{equation}
	
	Combining~\eqref{penaltym} with Lemma~\ref{lem:tr} gives:
	\begin{equation}
		\label{monot-1}
		\begin{aligned}
			&\dfrac{1}{\kappa}\int_{0}^{T} \langle \mathcal{N}\bm{u}^m_\kappa(t),\bm{u}^m_\kappa(t)\rangle_{\bm{V}^\ast(\Omega),\bm{V}(\Omega)}\dd t\\
			&=-\dfrac{1}{\kappa}\int_{0}^{T} \int_{\Gamma} \{[\bm{I}+u^m_{i,\kappa}(t)\bm{e}^i]\cdot\bm{q}\}^{-} (\bm{I}+u^m_{i,\kappa}\bm{e}^i)\cdot\bm{q} \dd\Gamma\dd t\\
			&\quad+\dfrac{1}{\kappa}\int_{0}^{T}\int_{\Gamma}\{[\bm{I}+u^m_{i,\kappa}(t)\bm{e}^i]\cdot\bm{q}\}^{-} (\bm{I}\cdot\bm{q})\dd\Gamma\dd t\\
			&\to -\int_{0}^{T} \int_{\Gamma} \chi_\kappa(t) (u_{i,\kappa}(t)\bm{e}^i\cdot\bm{q}) \dd\Gamma\dd t,\quad\textup{ as }m\to\infty.
		\end{aligned}
	\end{equation}
	
	Thanks to Lemma~\ref{lem:1}, the first convergence of~\eqref{unif} and the weak convergence~\eqref{penaltym}, Theorem~8.28 of~\cite{Leoni2017}, Theorem~8.62 of~\cite{Leoni2017} and the monotonicity of the operator $\tilde{\mathcal{N}}$ defined in~\eqref{Ntilde} and~\eqref{monot-1}, we are in a position to apply Theorem~12.5-2 of~\cite{Cia25} and, so, to obtain:
	\begin{equation}
		\label{penalty}
		\chi_\kappa=\kappa^{-1}\left\{[\bm{I}+u_{i,\kappa} \bm{e}^{i}]\cdot \bm{q}\right\}^{-} \in L^2((0,T)\times\Gamma).
	\end{equation}	
	
	We now verify that $\bm{u}_\kappa$ is a solution of the variational equations in Problem~\ref{problem2}. Let $\psi \in \mathcal D(0,T)$ and let $\tilde{m} \ge 1$ be any integer. For each $m \ge \tilde{m}$, we have
	\begin{equation}
		\label{density}
		\begin{aligned}
			&2\rho \int_{0}^{T} \int_{\Omega} \ddot{u}_{i,\kappa}^m(t) v_i\dd x \psi(t) \dd t+\int_{0}^{T} \int_{\Omega} A^{ijk\ell} e_{k\|\ell}(\bm{u}^m_\kappa(t)) e_{i\|j}(\bm{v}) \dd x \psi(t) \dd t\\
			&\quad+\int_{0}^{T} \int_{\Omega} B^{ijk\ell} e_{k\|\ell}(\dot{\bm{u}}^m_\kappa(t)) e_{i\|j}(\bm{v}) \dd x \psi(t) \dd t\\
			&\quad-\dfrac{1}{\kappa}\int_{0}^{T} \int_{\Gamma} \left(\left\{[\bm{I}+u^{m}_{i,\kappa}(t)\bm{e}^{i}]\cdot \bm{q}\right\}^{-}\right)  (v_{i}\bm{e}^{i} \cdot \bm{q}) \dd \Gamma\psi(t)\dd t\\
			&=\int_{0}^{T} \int_{\Omega} f^{i}(t) v_{i} \dd x \psi(t) \dd t,
		\end{aligned}
	\end{equation}
	for all $\bm{v} \in \bm{E}^{\tilde{m}} $. Thanks to Lemma~\ref{lem:1}, the convergence process~\eqref{c1}, \eqref{penalty}, the arbitrariness of $\psi\in \mathcal{D}(0,T)$, the fact that 
	\begin{equation*}
		\overline{\bigcup_{\tilde{m} \ge 1}\bm{E}^{\tilde{m}}}^{\|\cdot\|_{\bm{V}(\Omega)}} = \bm{V}(\Omega),
	\end{equation*}
	and the continuity of the operators $\mathcal{A}$ (Lemma~\ref{lem:2}) and $\mathcal{B}$ (Lemma~\ref{lem:3}), we obtain that a passage to the limit as $m \to \infty$ in~\eqref{density} gives that $\bm{u}_\kappa$ is a solution to the penalised variational equations in Problem~\ref{problem1}.
	
	The last property that we have to check is the validity of the initial conditions for $\bm{u}_\kappa$.
	Let us introduce the operator $\bm{L}_0:\mathcal{C}^0([0,T];\bm{V}(\Omega)) \to \bm{V}(\Omega)$ defined in a way such that $\bm{L}_0(\bm{v}):=\bm{v}(0)$. Such an operator $\bm{L}_0$ turns out to be linear and continuous and, therefore, by the first convergence of~\eqref{unif}, we immediately obtain:
	$$
	\bm{u}_\kappa(0)=\bm{u}_0, \quad\textup{ in }\bm{V}(\Omega).
	$$
	
	Similarly, let us introduce the operator $\bm{L}_1: \mathcal{C}^0([0,T];\bm{V}^\ast(\Omega)) \to \bm{V}^\ast(\Omega)$ defined in a way such that $\bm{L}_1(\bm{v}):=\bm{v}(0)$. Such an operator $\bm{L}_1$ turns out to be linear and continuous and since $\bm{u}^m_1$ is the projection of $\bm{u}_1$ onto $\bm{E}^m$, the second convergence of~\eqref{unif} gives:
	$$
	\bm{u}^m_1 \rightharpoonup \dot{\bm{u}}_\kappa(0)=\bm{u}_1, \quad\textup{ in }\bm{V}^\ast(\Omega) \textup{ as }m\to\infty.
	$$
	
	We thus obtain that $\bm{u}_\kappa$ is a solution of Problem~\ref{problem2}, and the proof is complete.
\end{proof}

The next step consists in showing that the family of solutions $\{\bm{u}_\kappa\}_{\kappa>0}$ of Problem~\ref{problem1} admits a subsequence $\{\bm{u}_{\kappa_n}\}_{n=1}^\infty$ (where $\{\kappa_n\}_{n=1}^\infty$ denotes a sequence of positive real numbers that converges to zero as $n\to\infty$) that converges in some suitable sense to a solution of Problem~\ref{problem0}.

Thanks to the energy estimates recovered in the proof of Theorem~\ref{th:1}, we can extract further compactness, as asserted in the following lemma.

\begin{lemma}
	\label{lem:4}
	Consider the sequence $\{\bm{u}_\kappa\}_{\kappa>0}$, where each $\bm{u}_\kappa$ is a solution of Problem~\ref{problem1}.
	Up to passing to a subsequence $\{\bm{u}_{\kappa_n}\}_{n=1}^\infty$ where $\{\kappa_n\}_{n=1}^\infty$ is such that $\kappa_n\to 0^+$ as $n\to\infty$, the following convergences hold:
	\begin{equation}
		\label{convproc1}
		\begin{aligned}
			\bm{u}_{\kappa_n} &\wsc \bm{u},\quad\textup{ in }L^\infty(0,T;\bm{V}(\Omega)) \textup{ as }n\to\infty,\\
			\dot{\bm{u}}_{\kappa_n} &\wsc \dot{\bm{u}},\quad\textup{ in }L^\infty(0,T;\bm{L}^2(\Omega)) \textup{ as }n\to\infty,\\
			\dot{\bm{u}}_{\kappa_n} &\rightharpoonup\dot{\bm{u}},\quad\textup{ in }L^2(0,T;\bm{V}(\Omega)) \textup{ as }n\to\infty,\\
			\{[\bm{I}+u_{i,\kappa}\bm{e}^i]\cdot\bm{q}\}^{-} &\to \{[\bm{I}+u_i\bm{e}^i]\cdot\bm{q}\}^{-}=0,\quad\textup{ in }L^2(0,T;\bm{L}^2(\Gamma)) \textup{ as }n\to\infty,\\
			\bm{u}_{\kappa_n} &\rightharpoonup \bm{u},\quad\textup{ in }\mathcal{C}^0([0,T];\bm{V}(\Omega)) \textup{ as }n\to\infty.
		\end{aligned}
	\end{equation}
	
	In particular, we obtain that $\bm{u} \in\mathcal{C}^0([0,T];\bm{U}(\Omega))$ and $\bm{u}(0) =\bm{u}_0 \in \bm{U}(\Omega)$.
\end{lemma}
\begin{proof}
	Testing the variational equations of Problem~\ref{problem1} at $\dot{\bm{u}}_\kappa(t)$ for a.a. $t\in (0,T)$, and proceeding similarly to item~(ii) in the proof of Theorem~\ref{th:1} we obtain energy estimates similar to~\eqref{pm2-2}, so that:
			\begin{equation}
				\label{pm5-2}
				\begin{aligned}
					\{\bm{u}_\kappa\}_{\kappa>0}&\textup{ is bounded in }L^\infty(0,T;\bm{V}(\Omega)) \textup{ independently of }\kappa,\\
					\{\dot{\bm{u}}_\kappa\}_{\kappa>0}&\textup{ is bounded in }L^\infty(0,T;\bm{L}^2(\Omega)) \cap L^2(0,T;\bm{V}(\Omega)) \textup{ independently of }\kappa,\\
					&\left\|\left\{[\bm{I}+u_{i,\kappa} \bm{e}^{i}]\cdot \bm{q}\right\}^{-}\right\|_{L^2(0,T;L^2(\Gamma))}\le \sqrt{\kappa T b e^T}.
				\end{aligned}
			\end{equation}
	
Therefore, the following convergences hold up to passing to a subsequence indexed over $\kappa_n$, where $\{\kappa_n\}_{n=1}^\infty$ denotes a sequence of positive real numbers such that $\kappa_n\to 0^+$ as $n\to\infty$:
\begin{equation*}
	\begin{aligned}
		\bm{u}_{\kappa_n} &\wsc \bm{u},\quad\textup{ in }L^\infty(0,T;\bm{V}(\Omega)) \textup{ as }n\to\infty,\\
		\dot{\bm{u}}_{\kappa_n} &\wsc \dot{\bm{u}},\quad\textup{ in }L^\infty(0,T;\bm{L}^2(\Omega)) \textup{ as }n\to\infty,\\
		\dot{\bm{u}}_{\kappa_n} &\rightharpoonup\dot{\bm{u}},\quad\textup{ in }L^2(0,T;\bm{V}(\Omega)) \textup{ as }n\to\infty.
	\end{aligned}
\end{equation*}
	
The convergence $\bm{u}_{\kappa_n}\rightharpoonup\bm{u}$ in $\mathcal{C}^0([0,T];\bm{V}(\Omega))$ as $n\to\infty$ holds thanks to Theorem~10.1.20~in~\cite{YP} (see also~\cite{GasPap2006}). The convergence $\bm{u}_{\kappa_n}\to\bm{u}$ in $\mathcal{C}^0([0,T];\bm{L}^2(\Omega))$ as $n\to\infty$ holds thanks to the Aubin-Lions-Simon theorem (Theorem~\ref{th:ALS}). Let us now observe that an application of~\eqref{pm6} gives
\begin{equation*}
	\begin{aligned}
		&-\int_{0}^{T} \int_{\Gamma} \{[\bm{I}+u_{i,\kappa_n}(t)\bm{e}^i]\cdot\bm{q}\}^{-} u_{i,\kappa_n}(t)\bm{e}^i\cdot\bm{q} \dd \Gamma\dd t\\
		&=-\int_{0}^{T} \int_{\Gamma} \{[\bm{I}+u_{i,\kappa_n}(t)\bm{e}^i]\cdot\bm{q}\}^{-} \left([\bm{I}+u_{i,\kappa_n}(t)\bm{e}^i]\cdot\bm{q}\right) \dd \Gamma\dd t+\int_{0}^{T} \int_{\Gamma} \{[\bm{I}+u_{i,\kappa_n}(t)\bm{e}^i]\cdot\bm{q}\}^{-} (\bm{I}\cdot\bm{q}) \dd\Gamma\dd t\\
		&=-\int_{0}^{T} \int_{\Gamma} \{[\bm{I}+u_{i,\kappa_n}(t)\bm{e}^i]\cdot\bm{q}\}^{-} \{[\bm{I}+u_{i,\kappa_n}(t)\bm{e}^i]\cdot\bm{q}\}^{+} \dd \Gamma\dd t\\
		&\quad+\int_{0}^{T} \int_{\Gamma} \{[\bm{I}+u_{i,\kappa_n}(t)\bm{e}^i]\cdot\bm{q}\}^{-} \{[\bm{I}+u_{i,\kappa_n}(t)\bm{e}^i]\cdot\bm{q}\}^{-} \dd \Gamma\dd t
		+\int_{0}^{T} \int_{\Gamma} \{[\bm{I}+u_{i,\kappa_n}(t)\bm{e}^i]\cdot\bm{q}\}^{-} (\bm{I}\cdot\bm{q}) \dd\Gamma\dd t\\
		&=\int_{0}^{T} \int_{\Gamma} \left|\{[\bm{I}+u_{i,\kappa_n}(t)\bm{e}^i]\cdot\bm{q}\}^{-}\right|^2 \dd \Gamma\dd t+\int_{0}^{T} \int_{\Gamma} \{[\bm{I}+u_{i,\kappa_n}(t)\bm{e}^i]\cdot\bm{q}\}^{-} (\bm{I}\cdot\bm{q}) \dd\Gamma\dd t \to 0,
	\end{aligned}
\end{equation*}
as $n\to\infty$. The Lipschitz continuity and monotonicity of the operator $-\{\cdot\}^{-}$ (Lemma~\ref{lem:1}) put us in position to apply the corollary to the Minty-Browder theorem (Theorem~12.5-2 in~\cite{Cia25}) so as to infer that
\begin{equation*}
	\{[\bm{I}+u_i\bm{e}^i]\cdot\bm{q}\}^{-} = 0,\quad\textup{ in }L^2(0,T;L^2(\Gamma)),
\end{equation*}
as well as the validity of the fourth convergence in~\eqref{convproc1}. In particular, the convergence $\bm{u}_{\kappa_n}\rightharpoonup\bm{u}$ in $\mathcal{C}^0([0,T];\bm{V}(\Omega))$ as $n\to\infty$ obtained beforehand as a result of Theorem~10.1.20 of~\cite{YP} implies that $\bm{u}\in\mathcal{C}^0([0,T];\bm{U}(\Omega))$ and that $\bm{u}(0)=\bm{u}_0$. This completes the proof.
\end{proof}

The next lemma aims to verify that the limit $\bm{u}$ introduced in Lemma~\ref{lem:4} satisfies the initial condition for the velocity.
Due to the limited regularity of the limit $\ddot{\bm{u}}$, we need to restrict our analysis to applied body forces $\bm{f}=(f^i)\in H^1(0,T;\bm{L}^2(\Omega))$.

\begin{lemma}
	\label{lem:5}
	Assume that $\bm{f}=(f^i)\in H^1(0,T;\bm{L}^2(\Omega))$. Then the limit $\dot{\bm{u}}\in L^\infty(0,T;\bm{L}^2(\Omega))\cap\mathcal{C}^0([0,T];\bm{H}^{-1}(\Omega))$ and is such that $\dot{\bm{u}}(0)=\bm{u}_1$ and $\dot{\bm{u}}(T)\in \bm{L}^2(\Omega)$.
	Besides, up to passing to a suitable subsequence $\{\dot{\bm{u}}_{\kappa_n}\}_{n=1}^\infty$, where $\{\kappa_n\}_{n=1}^\infty$ is a sequence such that $\kappa_n\to 0^+$ as $n\to\infty$, it results that:
	\begin{equation*}
		\begin{aligned}
			\dot{\bm{u}}_{\kappa_n}\to\dot{\bm{u}}&,\quad\textup{ in }L^2(0,T;\bm{L}^2(\Omega)) \textup{ as }n\to\infty,\\
			\dot{\bm{u}}_{\kappa_n}(T)\rightharpoonup\dot{\bm{u}}(T)&,\quad\textup{ in }\bm{L}^2(\Omega)\textup{ as }n\to\infty.
		\end{aligned}
	\end{equation*}
\end{lemma}
\begin{proof}
	For a.a. $t\in(0,T)$, and for all $\kappa>0$, any solution $\bm{u}_\kappa$ for Problem~\ref{problem1} satisfies the following equation:
	\begin{equation}
		\label{important-1}
		2\rho\ddot{\bm{u}}_\kappa(t)+(\tilde{\mathcal{A}}\bm{u}_\kappa)(t)+(\tilde{\mathcal{B}}\dot{\bm{u}}_\kappa)(t)+\dfrac{1}{\kappa}(\tilde{\mathcal{N}}\bm{u}_\kappa)(t)=\bm{f}(t),\quad\textup{ in }\bm{V}^\ast(\Omega).
	\end{equation}
	
	Testing~\eqref{important-1} along any element $\bm{v} \in \bm{H}^1_0(\Omega)$, we obtain that the penalty term vanishes. Therefore, the function $\bm{u}_\kappa$ satisfies the following equation
	\begin{equation}
		\label{important-2}
		\ddot{\bm{u}}_\kappa(t)+(\tilde{\mathcal{A}}\bm{u}_\kappa)(t)+(\tilde{\mathcal{B}}\dot{\bm{u}}_\kappa)(t)=\bm{f}(t),\quad\textup{ in }\bm{H}^{-1}(\Omega),
	\end{equation}
	in the sense of distributions in $(0,T)$. Note that~\eqref{important-2} is a linear equation; this implies that
	\begin{equation*}
		\{\ddot{\bm{u}}_\kappa\}_{\kappa>0} \textup{ is bounded in } L^2(0,T;\bm{H}^{-1}(\Omega)) \textup{ independently of }\kappa,
	\end{equation*}
	so that, up to passing to a suitable subsequence $\{\ddot{\bm{u}}_{\kappa_n}\}_{n=1}^\infty$ where $\{\kappa_n\}_{n=1}^\infty$ is a sequence such that $\kappa_n\to 0^+$ as $n\to\infty$ we obtain:
	\begin{equation*}
		\ddot{\bm{u}}_{\kappa_n}\rightharpoonup\ddot{\bm{u}},\quad\textup{ in }L^2(0,T;\bm{H}^{-1}(\Omega))\textup{ as }n\to\infty.
	\end{equation*}
	
	Therefore, it is immediate to observe that $\dot{\bm{u}}_{\kappa_n}\rightharpoonup\dot{\bm{u}}$ in $\mathcal{C}^0([0,T];\bm{H}^{-1}(\Omega))$ as $n\to\infty$ and that:
	\begin{equation}
		\label{eq:u'T}
		\dot{\bm{u}}_{\kappa_n}(T)\rightharpoonup\dot{\bm{u}}(T),\quad\textup{ in }\bm{H}^{-1}(\Omega)\textup{ as }n\to\infty.
	\end{equation}
	
	The latter allows us to show that $\dot{\bm{u}}(0)=\bm{u}_1$ and $\dot{\bm{u}}(T)\in\bm{H}^{-1}(\Omega)$.
	Additionally, an application of the the Aubin-Lions-Simon theorem (Theorem~\ref{th:ALS}) gives:
	\begin{equation*}
		\dot{\bm{u}}_{\kappa_n}\to \dot{\bm{u}},\quad\textup{ in }L^2(0,T;\bm{L}^2(\Omega)) \textup{ as }n\to\infty.
	\end{equation*}
	
	For each $\kappa>0$ and for a.a. $0\le t\le T$, define the energy:
	\begin{equation}
		\label{Ek}
		\begin{aligned}
		&E_\kappa(t):=\rho\int_{\Omega}|\dot{\bm{u}}_\kappa(t)|^2\dd x+\dfrac{1}{2}\int_{\Omega}A^{ijk\ell}e_{k\|\ell}(\bm{u}_\kappa(t))e_{i\|j}(\bm{u}_\kappa(t))\dd x\\
		&\quad+\dfrac{1}{2\kappa}\int_{\Gamma}|\{[\bm{I}+u_{i,\kappa}(t)\bm{e}^i]\cdot\bm{q}\}^{-}|^2\dd\Gamma-\int_{\Omega}f^i(t)u_{i,\kappa}(t)\dd x.
		\end{aligned}
	\end{equation}
	
	Observe that:
	\begin{equation}
		\label{EkL1}
		\begin{aligned}
			&\int_{0}^{T}|E_\kappa(t)|\dd t\le\rho\|\dot{\bm{u}}_\kappa(t)\|_{\bm{L}^2(\Omega)}^2+\dfrac{\max_{i,j,k,\ell}\|A^{ijk\ell}\|_{\mathcal{C}^0(\overline{\Omega})}c_0^2}{2}\int_{0}^{T}\|\bm{u}_\kappa(t)\|_{\bm{V}(\Omega)}^2\dd t\\
			&\quad+\dfrac{1}{2\kappa}\int_{\Gamma}\left|\{[\bm{I}+u_{i,\kappa}(t)\bm{e}^i]\cdot\bm{q}\}^{-}\right|^2\dd\Gamma+\dfrac{1}{2}\int_{0}^{T}\left\{\sum_{i=1}^3\|f^i(t)\|_{L^2(\Omega)}^2\right\}\dd t\\
			&\quad+\dfrac{1}{2}\int_{0}^{T}\left\{\sum_{i=1}^3\|u_{i,\kappa}(t)\|_{L^2(\Omega)}^2\right\}\dd t.
		\end{aligned}
	\end{equation}
	
	Since the right-hand side of~\eqref{EkL1} is bounded independently of $\kappa$ thanks to~\eqref{pm5-2}, we infer that $E_\kappa\in L^1(0,T)$, and that the family $\{E_\kappa\}_{\kappa>0}$ is bounded in $L^1(0,T)$ independently of $\kappa$. Observe that, thanks to the regularity of $\bm{u}_\kappa$ and Corollary~10.1.26 of~\cite{YP}, it results that $E_\kappa$ is differentiable a.e. in $(0,T)$.
	
	Differentiating $E_\kappa$ and invoking the variational equations of Problem~\ref{problem1} tested at $\bm{v}=\dot{\bm{u}}_\kappa(t)$ give:
	\begin{equation}
		\label{eq:dEkappa0}
		\begin{aligned}
			&\dot{E}_\kappa(t)=\rho\dfrac{\dd}{\dd t}\left(\int_{\Omega}|\dot{\bm{u}}_\kappa(t)|^2\dd x\right)+\dfrac{1}{2}\dfrac{\dd}{\dd t}\left(\int_{\Omega}A^{ijk\ell}e_{k\|\ell}(\bm{u}_\kappa(t))e_{i\|j}(\bm{u}_\kappa(t))\dd x\right)\\
			&\quad+\dfrac{1}{2\kappa}\int_{\Gamma}\left|\{[\bm{I}+u_{i,\kappa}(t)\bm{e}^i]\cdot\bm{q}\}^{-}\right|^2\dd\Gamma-\int_{\Omega}f^i(t)\dot{u}_{i,\kappa}(t)\dd x-\int_{\Omega}\dot{f}^i(t)u_{i,\kappa}(t)\dd x\\
			&=-\int_{\Omega}B^{ijk\ell}e_{k\|\ell}(\dot{\bm{u}}_\kappa(t))e_{i\|j}(\dot{\bm{u}}_\kappa(t))\dd x-\int_{\Omega}\dot{f}^i(t)u_{i,\kappa}(t)\dd x.
		\end{aligned}
	\end{equation}
	
	Taking the absolute value in~\eqref{eq:dEkappa0} and employing Young's inequality~\cite{Young1912} gives:
	\begin{equation}
		\label{eq:dEkappa1}
		|\dot{E}_\kappa(t)|\le\left|\int_{\Omega}B^{ijk\ell}e_{k\|\ell}(\dot{\bm{u}}_\kappa(t))e_{i\|j}(\dot{\bm{u}}_\kappa(t))\dd x\right|+\dfrac{1}{2}\sum_{i=1}^3\|\dot{f}^i(t)\|_{L^2(\Omega)}^2+\dfrac{1}{2}\sum_{i=1}^3\|u_{i,\kappa}(t)\|_{L^2(\Omega)}^2.
	\end{equation}
	
	Observe that, thanks to~\eqref{pm5-2}, the right-hand side of~\eqref{eq:dEkappa1} is bounded independently of $\kappa$ showing that the family $\{\dot{E}_\kappa\}_{\kappa>0}$ is bounded in $L^1(0,T)$ independently of $\kappa$, and it thus results that $\dot{E}_\kappa\in L^1(0,T)$. Therefore, we obtain that $E_\kappa\in W^{1,1}(0,T)\subset \gls{AC}$.
	The continuity of $E_\kappa$ in $[0,T]$ and the continuity of the second, third and fourth term on the right-hand side of~\eqref{Ek} implies that the mapping
	\begin{equation}
		\label{eq:map-0}
		[0,T]\ni t\mapsto \int_{\Omega}|\dot{\bm{u}}_\kappa(t)|^2\dd x,
	\end{equation}
	is continuous and, thanks to~\eqref{convproc1}, is bounded in $[0,T]$ independently of $\kappa$. The boundedness of the mapping in~\eqref{eq:map-0} in $[0,T]$ independently of $\kappa$ is equivalent to stating that there exists a constant $C>0$ independent of $\kappa$ such that
	\begin{equation*}
		\max_{t\in [0,T]}\|\dot{\bm{u}}_\kappa(t)\|_{\bm{L}^2(\Omega)} \le C,\quad\textup{ for all }\kappa>0,
	\end{equation*}
	and, in particular, that $\{\dot{\bm{u}}_\kappa(T)\}_{\kappa>0}$ is bounded in $\bm{L}^2(\Omega)$ independently of $\kappa$. Therefore, up to passing to a suitable subsequence $\{\dot{\bm{u}}_{\kappa_n}(T)\}_{n=1}^\infty$ where $\{\kappa_n\}_{n=1}^\infty$ is a sequence such that $\kappa_n\to 0^+$ as $n\to\infty$ we obtain:
	\begin{equation*}
		\dot{\bm{u}}_{\kappa_n}(T)\rightharpoonup\bm{U}_T,\quad\textup{ in }\bm{L}^2(\Omega)\textup{ as }n\to\infty.
	\end{equation*}
	
	Since we have already observed in~\eqref{eq:u'T} that $\dot{\bm{u}}_{\kappa_n}(T)\rightharpoonup\dot{\bm{u}}(T)$ in $\bm{H}^{-1}(\Omega)$ as $n\to\infty$, the uniqueness of the weak limit gives that $\dot{\bm{u}}(T)=\bm{U}_T\in\bm{L}^2(\Omega)$.
\end{proof}

We are now ready to establish the main result of this paper: the existence of solutions for Problem~\ref{problem0}.

\begin{theorem}
	\label{exP0}
	Assume that $\bm{f}\in H^1(0,T;\bm{L}^2(\Omega))$. Then Problem~\ref{problem0} admits at least one solution.
\end{theorem}
\begin{proof}
	Fix $\bm{v}=(v_i) \in L^\infty(0,T;\bm{U}(\Omega))$ such that $\dot{\bm{v}}\in L^\infty(0,T;\bm{L}^2(\Omega))\cap L^2(0,T;\bm{V}(\Omega))$. For any given $t\in[0,T]$, let us test equation~\eqref{important-1} at the element $(\bm{v}(t)-\bm{u}_\kappa(t)) \in \bm{V}(\Omega)$.
	Thanks to the convergence $\dot{\bm{u}}_{\kappa_n}(T)\rightharpoonup\dot{\bm{u}}(T)$ in $\bm{L}^2(\Omega)$ as $n\to\infty$ established in Lemma~\ref{lem:5}, and thanks to the convergence $\dot{\bm{u}}_{\kappa_n}\rightharpoonup\dot{\bm{u}}$ in $L^2(0,T;\bm{L}^2(\Omega))$ as $n\to\infty$ established in Lemma~\ref{lem:4}, we obtain that:
	\begin{equation}
		\label{pezzo1}
		\begin{aligned}
			&\int_{0}^{T}\langle \ddot{\bm{u}}_{\kappa_n}(t),\bm{v}(t)\rangle_{\bm{V}^\ast(\Omega),\bm{V}(\Omega)} \dd t\\
			&= \int_{\Omega}\dot{\bm{u}}_{\kappa_n}(T)\cdot\bm{v}(T)\dd x-\int_{\Omega}\bm{u}_1\cdot\bm{v}(0)\dd x-\int_{0}^{T}\int_{\Omega}\dot{\bm{u}}_{\kappa_n}(t)\cdot\dot{\bm{v}}(t)\dd x\dd t\\
			&\to\int_{\Omega}\dot{\bm{u}}(T)\cdot\bm{v}(T)\dd x-\int_{\Omega}\bm{u}_1\cdot\bm{v}(0)\dd x-\int_{0}^{T}\int_{\Omega}\dot{\bm{u}}(t)\cdot\dot{\bm{v}}(t)\dd x\dd t,\quad\textup{ as }n\to\infty.
		\end{aligned}
	\end{equation}
	
	The convergence $\dot{\bm{u}}_{\kappa_n}(T)\rightharpoonup\dot{\bm{u}}(T)$ in $\bm{L}^2(\Omega)$ as $n\to\infty$ established in Lemma~\ref{lem:5}, the convergence $\dot{\bm{u}}_{\kappa_n}\to\dot{\bm{u}}$ in $L^2(0,T;\bm{L}^2(\Omega))$ as $n\to\infty$ established in Lemma~\ref{lem:5}, and the convergence $\bm{u}_{\kappa_n}\to\bm{u}$ in $\mathcal{C}^0([0,T];\bm{L}^2(\Omega))$ as $n\to\infty$ established in Lemma~\ref{lem:4} give:
	\begin{equation}
		\label{pezzo2}
		\begin{aligned}
			&-\int_{0}^{T} \langle\ddot{\bm{u}}_{\kappa_n}(t),\bm{u}_{\kappa_n}(t)\rangle_{\bm{V}^\ast(\Omega),\bm{V}(\Omega)} \dd t\\
			&=-\int_{\Omega}\dot{\bm{u}}_{\kappa_n}(T)\cdot\bm{u}_{\kappa_n}(T)\dd x + \int_{\Omega} \bm{u}_1 \cdot \bm{u}_0 \dd x +\|\dot{\bm{u}}_{\kappa_n}\|_{L^2(0,T;\bm{L}^2(\Omega))}^2\\
			&\to-\int_{\Omega}\dot{\bm{u}}(T)\cdot\bm{u}(T)\dd x + \int_{\Omega} \bm{u}_1 \cdot \bm{u}_0 \dd x+\|\dot{\bm{u}}\|_{L^2(0,T;\bm{L}^2(\Omega))}^2,\quad\textup{ as }n\to\infty.
		\end{aligned}
	\end{equation}
	
	The convergence $\dot{\bm{u}}_{\kappa_n}\rightharpoonup\dot{\bm{u}}$ in $L^2(0,T;\bm{V}(\Omega))$ as $n\to\infty$ established in~\eqref{convproc1} and the continuity of the operator $\mathcal{B}$ (Lemma~\ref{lem:3}) imply that:
	\begin{equation}
		\label{pezzo3a}
		\int_{0}^{T} \langle\mathcal{B}\dot{\bm{u}}_{\kappa_n}(t),\bm{v}(t)\rangle_{\bm{V}^\ast(\Omega),\bm{V}(\Omega)} \dd t
		\to\int_{0}^{T} \langle\mathcal{B}\dot{\bm{u}}(t),\bm{v}(t)\rangle_{\bm{V}^\ast(\Omega),\bm{V}(\Omega)} \dd t,\quad\textup{ as }n\to\infty.
	\end{equation}
	
	The convergence $\bm{u}_{\kappa_n} \rightharpoonup \bm{u}$ in $\mathcal{C}^0([0,T];\bm{V}(\Omega))$ as $n\to\infty$ established in~\eqref{convproc1} implies that:
	\begin{equation}
		\label{pezzo3}
		\begin{aligned}
			&\limsup_{n\to\infty}\left(-\int_{0}^{T} \langle\mathcal{B}\dot{\bm{u}}_\kappa(t),\bm{u}_\kappa(t)\rangle_{\bm{V}^\ast(\Omega),\bm{V}(\Omega)} \dd t\right)\\
			&\le-\dfrac{1}{2}\int_{\Omega}B^{ijk\ell} e_{k\|\ell}(\bm{u}(T)) e_{i\|j}(\bm{u}(T)) \dd x+\dfrac{1}{2}\int_{\Omega}B^{ijk\ell} e_{k\|\ell}(\bm{u}_0) e_{i\|j}(\bm{u}_0) \dd x.
		\end{aligned}
	\end{equation}
	
	The convergence $\bm{u}_{\kappa_n} \rightharpoonup \bm{u}$ in $L^2(0,T;\bm{V}(\Omega))$ as $n\to\infty$ established in~\eqref{convproc1} and the continuity of the operator $\mathcal{A}$ (Lemma~\ref{lem:2}) imply that:
	\begin{equation}
		\label{pezzo4a}
		\int_{0}^{T} \langle\mathcal{A}\bm{u}_{\kappa_n}(t),\bm{v}(t)\rangle_{\bm{V}^\ast(\Omega),\bm{V}(\Omega)} \dd t
		\to\int_{0}^{T} \langle\mathcal{A}\bm{u}(t),\bm{v}(t)\rangle_{\bm{V}^\ast(\Omega),\bm{V}(\Omega)} \dd t,\quad\textup{ as }n\to\infty.
	\end{equation}
	
	The convergence $\bm{u}_{\kappa_n}\rightharpoonup\bm{u}$ in $L^2(0,T;\bm{V}(\Omega))$ as $n\to\infty$ established in~\eqref{convproc1} implies that:
	\begin{equation}
		\label{pezzo4}
		\begin{aligned}
			&\limsup_{n\to\infty}\left(-\int_{0}^{T} \langle\mathcal{A}\bm{u}_{\kappa_n}(t),\bm{u}_{\kappa_n}(t)\rangle_{\bm{V}^\ast(\Omega),\bm{V}(\Omega)} \dd t\right) \le
			-\int_{0}^{T} \int_{\Omega} A^{ijk\ell} e_{k\|\ell}(\bm{u}(t)) e_{i\|j}(\bm{u}(t)) \dd x \dd t.
		\end{aligned}
	\end{equation}
	
	For what concerns the penalty term, we observe that the monotonicity of $-\{\cdot\}^{-}$ (Lemma~\ref{lem:1}) gives
	\begin{equation}
		\label{pezzo5}
		-\dfrac{1}{\kappa}\int_{0}^{T}\int_{\Gamma}\{[\bm{I}+u_{i,\kappa}(t)\bm{e}^i]\cdot\bm{q}\}^{-}\left([\bm{I}+v_i(t)\bm{e}^i]\cdot\bm{q}-[\bm{I}+u_{i,\kappa}(t)\bm{e}^i]\cdot\bm{q}\right) \dd\Gamma\dd t \le 0,
	\end{equation}
	for all $\kappa>0$.
	
	Finally, the convergence $\bm{u}_{\kappa_n}\rightharpoonup\bm{u}$ in $L^2(0,T;\bm{L}^2(\Omega))$ as $n\to\infty$ established in~\eqref{convproc1} implies that:
	\begin{equation}
		\label{pezzo6}
		-\int_{0}^{T}\int_{\Omega} f^i(t) u_{i,\kappa_n}(t) \dd x\dd t\to-\int_{0}^{T}\int_{\Omega} f^i(t) u_i(t) \dd x\dd t,\quad\textup{ as }n\to\infty.
	\end{equation}
	
	Combining~\eqref{pezzo1}--\eqref{pezzo6}, we obtain that the limit $\bm{u}$ satisfies the following hyperbolic variational inequalities:
	\begin{equation*}
		\begin{aligned}
			&2\rho\int_{\Omega}\dot{\bm{u}}(T)\cdot(\bm{v}(T)-\bm{u}(T))\dd x-2\rho\int_{\Omega}\bm{u}_1\cdot(\bm{v}(0)-\bm{u}_0)\dd x-2\rho\int_{0}^{T}\int_{\Omega}\dot{\bm{u}}(t)\cdot(\dot{\bm{v}}(t)-\dot{\bm{u}}(t))\dd x\dd t\\
			&\quad+\int_{0}^{T}\int_{\Omega}A^{ijk\ell} e_{k\|\ell}(\bm{u}(t)) e_{i\|j}(\bm{v}(t)-\bm{u}(t))\dd x\dd t
			+\int_{0}^{T} \int_{\Omega}B^{ijk\ell} e_{k\|\ell}(\dot{\bm{u}}(t)) e_{i\|j}(\bm{v}(t))\dd x\dd t\\
			&\quad-\dfrac{1}{2}\int_{\Omega}B^{ijk\ell} e_{k\|\ell}(\bm{u}(T)) e_{i\|j}(\bm{u}(T)) \dd x+\dfrac{1}{2}\int_{\Omega}B^{ijk\ell} e_{k\|\ell}(\bm{u}_0) e_{i\|j}(\bm{u}_0) \dd x\\
			&\ge \int_{0}^{T}\int_{\Omega}f^i(t)(v_i(t)-u_i(t))\dd x\dd t,
		\end{aligned}
	\end{equation*}
	which is exactly the governing model of Problem~\ref{problem0}. This completes the proof.
\end{proof}

Finally, we observe that if we assume that there exists a \emph{shut-down time} $T_0>0$ such that $\bm{f}(t)=\bm{0}$ for a.a. $t>T_0$, we are able to establish that, for \emph{certain materials}, the linearly elastic viscoelastic body returns to its reference configuration when the applied body force is lifted and that, in agreement with the theory of Kelvin-Voigt materials~\cite{Christensen12,Slonimsky67}, the decay is exponential.

\begin{theorem}
	\label{th:decay}
	Assume that there exists $0<T_0<T$ such that $\bm{f}(t)=\bm{0}$ for a.a. $T_0<t<T$. Assume that the following \emph{smallness condition} holds:
	\begin{equation}
		\label{eq:smallness}
		\dfrac{\rho}{C_e^{(1)}\max_{i,j,k,\ell}\|B^{ijk\ell}\|_{\mathcal{C}^0(\overline{\Omega})}}\le\dfrac{1}{2c_0^2C_e^{(2)}}.
	\end{equation}
	
	For each $\kappa>0$ and for a.a. $T_0\le t\le T$, define the energy:
	\begin{equation*}
		E_\kappa(t):=\rho\int_{\Omega}|\dot{\bm{u}}_\kappa(t)|^2\dd x+\dfrac{1}{2}\int_{\Omega}A^{ijk\ell}e_{k\|\ell}(\bm{u}_\kappa(t))e_{i\|j}(\bm{u}_\kappa(t))\dd x
		+\dfrac{1}{2\kappa}\int_{\Gamma}|\{[\bm{I}+u_{i,\kappa}(t)\bm{e}^i]\cdot\bm{q}\}^{-}|^2\dd\Gamma.
	\end{equation*}
	
	Then, $E_\kappa\in AC([T_0,T])$ and there exists a constant $\tilde{C}=\tilde{C}(\Omega,\lambda,\mu,\xi,\theta,\rho)>0$ and a constant $c_{{\textup{d}}}=c_{{\textup{d}}}(\Omega,\lambda,\mu,\xi,\theta,\rho)>0$ such that:
	\begin{equation}
		\label{eq:bound}
		\|\bm{u}_\kappa(t)\|_{\bm{V}(\Omega)}+\|\dot{\bm{u}}_\kappa(t)\|_{\bm{L}^2(\Omega)}\le\sqrt{\tilde{C}}\exp\left(-\dfrac{c_{{\textup{d}}}}{2}(t-T_0)\right),
	\end{equation}
	for a.a. $T_0<t< T$ and all $\kappa>0$.
\end{theorem}
\begin{proof}
	Let $\kappa>0$ be given. To show that the function $E_\kappa$ is of class $L^1(T_0,T)$ we proceed as in Lemma~\ref{lem:5}. Observe that, thanks to Corollary~10.1.26 of~\cite{YP}, it results that $E_\kappa$ is differentiable a.e. in $(T_0,T)$ and that $\dot{E}_\kappa\in L^1(T_0,T)$.
	
	Differentiating $E_\kappa$ and invoking the variational equations of Problem~\ref{problem1} tested at $\bm{v}=\dot{\bm{u}}_\kappa(t)$ give:
	\begin{equation}
		\label{eq:dEkappa}
		\begin{aligned}
			&\dot{E}_\kappa(t)=\rho\dfrac{\dd}{\dd t}\left(\int_{\Omega}|\dot{\bm{u}}_\kappa(t)|^2\dd x\right)+\dfrac{1}{2}\dfrac{\dd}{\dd t}\left(\int_{\Omega}A^{ijk\ell}e_{k\|\ell}(\bm{u}_\kappa(t))e_{i\|j}(\bm{u}_\kappa(t))\dd x\right)\\
			&\quad+\dfrac{1}{2\kappa}\dfrac{\dd}{\dd t}\left(\int_{\Gamma}\left|\{[\bm{I}+u_{i,\kappa}(t)\bm{e}^i]\cdot\bm{q}\}^{-}\right|^2\dd\Gamma\right)\\
			&=-\int_{\Omega}B^{ijk\ell}e_{k\|\ell}(\dot{\bm{u}}_\kappa(t))e_{i\|j}(\dot{\bm{u}}_\kappa(t))\dd x\le 0,
		\end{aligned}
	\end{equation}
	showing that $\dot{E}_\kappa\in L^1(T_0,T)$. Therefore,in the same spirit as Lemma~\ref{lem:5}, we obtain that $E_\kappa\in W^{1,1}(T_0,T)\subset AC([T_0,T])$ and that $E_\kappa$ decreases in $[T_0,T]$. The continuity of $E_\kappa$ in $[T_0,T]$ implies that the mapping
	\begin{equation}
		\label{eq:map}
		[0,T]\ni t\mapsto \int_{\Omega}|\dot{\bm{u}}_\kappa(t)|^2\dd x,
	\end{equation}
	is continuous and, thanks to~\eqref{pm5-2}, is bounded in $[T_0,T]$ independently of $\kappa$.
	
	In correspondence of $\varepsilon>0$, that we will determine later, we denote and define the modified energy $H_{\kappa,\varepsilon}$ by:
	\begin{equation*}
		H_{\kappa,\varepsilon}(t):=E_\kappa(t)+\varepsilon\rho\int_{\Omega}\dot{\bm{u}}_\kappa(t)\cdot\bm{u}_\kappa(t)\dd x,\quad\textup{ for a.a. }t\in (T_0,T).
	\end{equation*}
	
	Thanks to the regularity of $\bm{u}_\kappa$ in Lemma~\ref{lem:5} and thanks to the regularity of $E_\kappa$ observed beforehand, we can differentiate $H_{\kappa,\varepsilon}$ with respect to the variable $t$, getting
	\begin{equation}
		\label{step1}
		\begin{aligned}
			&\dot{H}_{\kappa,\varepsilon}(t)=\dot{E}_\kappa(t)+\varepsilon\rho\langle\ddot{\bm{u}}_\kappa(t),\bm{u}_\kappa(t)\rangle_{\bm{V}^\ast(\Omega),\bm{V}(\Omega)}+\varepsilon\rho\int_{\Omega}|\dot{\bm{u}}_\kappa(t)|^2\dd x\\
			&=-\int_{\Omega}B^{ijk\ell}e_{k\|\ell}(\dot{\bm{u}}_\kappa(t))e_{i\|j}(\dot{\bm{u}}_\kappa(t))\dd x
			-\dfrac{\varepsilon}{2}\int_{\Omega}A^{ijk\ell}e_{k\|\ell}(\bm{u}_\kappa(t))e_{i\|j}(\bm{u}_\kappa(t))\dd x\\
			&\quad-\dfrac{\varepsilon}{2}\int_{\Omega}B^{ijk\ell}e_{k\|\ell}(\dot{\bm{u}}_\kappa(t))e_{i\|j}(\bm{u}_\kappa(t))\dd x
			-\dfrac{\varepsilon}{2\kappa}\int_{\Gamma}\left|\{[\bm{I}+u_{i,\kappa}(t)\bm{e}^i]\cdot\bm{q}\}^{-}\right|^2\dd\Gamma\\
			&\quad-\dfrac{\varepsilon}{2\kappa}\int_{\Gamma}\{[\bm{I}+u_{i,\kappa}(t)\bm{e}^i]\cdot\bm{q}\}^{-}(\bm{I}\cdot\bm{q})\dd\Gamma
			+\varepsilon\rho\int_{\Omega}|\dot{\bm{u}}_\kappa(t)|^2\dd x\\
			&\le-\int_{\Omega}B^{ijk\ell}e_{k\|\ell}(\dot{\bm{u}}_\kappa(t))e_{i\|j}(\dot{\bm{u}}_\kappa(t))\dd x
			-\dfrac{\varepsilon}{2}\int_{\Omega}A^{ijk\ell}e_{k\|\ell}(\bm{u}_\kappa(t))e_{i\|j}(\bm{u}_\kappa(t))\dd x\\
			&\quad+\dfrac{\varepsilon}{2}\left[\left(\int_{\Omega}B^{ijk\ell}e_{k\|\ell}(\dot{\bm{u}}_\kappa(t))e_{i\|j}(\dot{\bm{u}}_\kappa(t))\dd x\right)^{1/2}\left(\int_{\Omega}B^{ijk\ell}e_{k\|\ell}(\bm{u}_\kappa(t))e_{i\|j}(\bm{u}_\kappa(t))\dd x\right)^{1/2}\right]\\
			&\quad-\dfrac{\varepsilon}{2\kappa}\int_{\Gamma}\left|\{[\bm{I}+u_{i,\kappa}(t)\bm{e}^i]\cdot\bm{q}\}^{-}\right|^2\dd\Gamma
			+\varepsilon\rho\int_{\Omega}|\dot{\bm{u}}_\kappa(t)|^2\dd x\\
			&\le\left(-1+\dfrac{\varepsilon\eta}{4}\right)\int_{\Omega}B^{ijk\ell}e_{k\|\ell}(\dot{\bm{u}}_\kappa(t))e_{i\|j}(\dot{\bm{u}}_\kappa(t))\dd x
			-\dfrac{\varepsilon}{2}\int_{\Omega}A^{ijk\ell}e_{k\|\ell}(\bm{u}_\kappa(t))e_{i\|j}(\bm{u}_\kappa(t))\dd x\\
			&\quad+\dfrac{\varepsilon}{4\eta}\int_{\Omega}B^{ijk\ell}e_{k\|\ell}(\bm{u}_\kappa(t))e_{i\|j}(\bm{u}_\kappa(t))\dd x
			-\dfrac{\varepsilon}{2\kappa}\int_{\Gamma}\left|\{[\bm{I}+u_{i,\kappa}(t)\bm{e}^i]\cdot\bm{q}\}^{-}\right|^2\dd\Gamma
			+\varepsilon\rho\int_{\Omega}|\dot{\bm{u}}_\kappa(t)|^2\dd x\\
			&\le\left(-1+\dfrac{\varepsilon\eta}{4}\right)\int_{\Omega}B^{ijk\ell}e_{k\|\ell}(\dot{\bm{u}}_\kappa(t))e_{i\|j}(\dot{\bm{u}}_\kappa(t))\dd x+\varepsilon\rho\int_{\Omega}|\dot{\bm{u}}_\kappa(t)|^2\dd x\\
			&\quad-\dfrac{\varepsilon}{2\kappa}\int_{\Gamma}\left|\{[\bm{I}+u_{i,\kappa}(t)\bm{e}^i]\cdot\bm{q}\}^{-}\right|^2\dd\Gamma\\
			&\quad+\dfrac{\varepsilon}{2}\left(\dfrac{C_e^{(1)}\max_{i,j,k,\ell}\|B^{ijk\ell}\|_{\mathcal{C}^0(\overline{\Omega})}}{\eta }-1\right)\int_{\Omega}A^{ijk\ell}e_{k\|\ell}(\bm{u}_\kappa(t))e_{i\|j}(\bm{u}_\kappa(t))\dd x,
		\end{aligned}
	\end{equation}
	for a.a. $T_0<t<T$, where the second equality holds thanks to~\eqref{eq:dEkappa} and the variational equations of Problem~\ref{problem1} tested at $\bm{u}_\kappa(t)$, the first estimate is due to the Cauchy-Schwarz inequality (cf., e.g., Proposition~5.3 in~\cite{CanDAp14}), the second estimate is due to Young's inequality~\cite{Young1912} for all $\eta>0$, and the third estimate is due to Korn's inequality (Theorem~\ref{KornCart}), the boundedness of the fourth order three-dimensional viscosity tensor $\{B^{ijk\ell}\}$, and the uniform positive-definiteness of the fourth order three-dimensional elasticity tensor $\{A^{ijk\ell}\}$. Choosing $\eta:=2C_e^{(1)}\max_{i,j,k,\ell}\|B^{ijk\ell}\|_{\mathcal{C}^0(\overline{\Omega})}$ and $0<\varepsilon\le \frac{2}{\eta}$ in~\eqref{step1} gives
	\begin{equation}
		\label{step2}
		\begin{aligned}
			&\dot{H}_{\kappa,\varepsilon}(t)\le-\dfrac{1}{2}\int_{\Omega}B^{ijk\ell}e_{k\|\ell}(\dot{\bm{u}}_\kappa(t))e_{i\|j}(\dot{\bm{u}}_\kappa(t))\dd x+\varepsilon\rho\int_{\Omega}|\dot{\bm{u}}_\kappa(t)|^2\dd x\\
			&\quad-\dfrac{\varepsilon}{2\kappa}\int_{\Gamma}\left|\{[\bm{I}+u_{i,\kappa}(t)\bm{e}^i]\cdot\bm{q}\}^{-}\right|^2\dd\Gamma-\dfrac{\varepsilon}{4}\int_{\Omega}A^{ijk\ell}e_{k\|\ell}(\bm{u}_\kappa(t))e_{i\|j}(\bm{u}_\kappa(t))\dd x\\
			&\le-\left(\dfrac{1}{2C_e^{(2)}c_0^2}-\varepsilon\rho\right)\int_{\Omega}|\dot{\bm{u}}_\kappa(t)|^2\dd x-\dfrac{\varepsilon}{4}\int_{\Omega}A^{ijk\ell}e_{k\|\ell}(\bm{u}_\kappa(t))e_{i\|j}(\bm{u}_\kappa(t))\dd x\\
			&\quad-\dfrac{\varepsilon}{2\kappa}\int_{\Gamma}\left|\{[\bm{I}+u_{i,\kappa}(t)\bm{e}^i]\cdot\bm{q}\}^{-}\right|^2\dd\Gamma\\
			&\le-\left(\dfrac{1}{2c_0^2C_e^{(2)}}-\dfrac{\rho}{C_e^{(1)}\max_{i,j,k,\ell}\|B^{ijk\ell}\|_{\mathcal{C}^0(\overline{\Omega})}}\right)\int_{\Omega}|\dot{\bm{u}}_\kappa(t)|^2\dd x\\
			&\quad-\dfrac{1}{2C_e^{(1)}\max_{i,j,k,\ell}\|B^{ijk\ell}\|_{\mathcal{C}^0(\overline{\Omega})}}\int_{\Omega}A^{ijk\ell}e_{k\|\ell}(\bm{u}_\kappa(t))e_{i\|j}(\bm{u}_\kappa(t))\dd x\\
			&\quad-\dfrac{1}{2C_e^{(1)}\kappa\max_{i,j,k,\ell}\|B^{ijk\ell}\|_{\mathcal{C}^0(\overline{\Omega})}}\int_{\Gamma}\left|\{[\bm{I}+u_{i,\kappa}(t)\bm{e}^i]\cdot\bm{q}\}^{-}\right|^2\dd\Gamma\\
			&\le-\min\left\{\dfrac{1}{\rho}\left(\dfrac{1}{2c_0^2C_e^{(2)}}-\dfrac{\rho }{C_e^{(1)}\max_{i,j,k,\ell}\|B^{ijk\ell}\|_{\mathcal{C}^0(\overline{\Omega})}}\right),\dfrac{1}{C_e^{(1)}\max_{i,j,k,\ell}\|B^{ijk\ell}\|_{\mathcal{C}^0(\overline{\Omega})}}\right\}E_\kappa(t),
		\end{aligned}
	\end{equation}
	and we observe that, thanks to the \emph{smallness condition}~\eqref{eq:smallness}, the coefficient of the first term in the second-last inequality is non-positive.
	
	Let us now observe that, on the one hand, an application of the Cauchy-Schwarz inequality (cf., e.g., Proposition~5.3 in~\cite{CanDAp14}) Young's inequality~\cite{Young1912}, Korn's inequality (Theorem~\ref{KornCart}) and the uniform positive-definiteness of the fourth order three-dimensional elasticity tensor $\{A^{ijk\ell}\}$ gives that for a.a. $T_0<t<T$:
	\begin{equation}
		\label{eq:equiv-1}
		\begin{aligned}
			&H_{\kappa,\varepsilon}(t)\le E_\kappa(t)+\dfrac{\varepsilon\rho}{2}\|\dot{\bm{u}}_\kappa(t)\|_{\bm{L}^2(\Omega)}^2+\dfrac{\varepsilon\rho}{2}\|\bm{u}_\kappa(t)\|_{\bm{V}(\Omega)}^2\\
			&\le\rho\left(1+\dfrac{\varepsilon}{2}\right)\int_{\Omega}|\dot{\bm{u}}_\kappa(t)|^2\dd x+\dfrac{1}{2}\left(1+\varepsilon\rho c_0^2C_e^{(1)}\right)\int_{\Omega}A^{ijk\ell}e_{k\|\ell}(\bm{u}_\kappa(t))e_{i\|j}(\bm{u}_\kappa(t))\dd x\\
			&\quad+\dfrac{1}{2\kappa}\int_{\Gamma}\left|\{[\bm{I}+u_{i,\kappa}(t)\bm{e}^i]\cdot\bm{q}\}^{-}\right|^2\dd\Gamma
			\le\max\left\{1+\dfrac{\varepsilon}{2},1+\varepsilon\rho c_0^2C_e^{(1)}\right\}E_\kappa(t).
		\end{aligned}
	\end{equation}
	
	On the other hand, choosing 
	$$
	0<\varepsilon<\min\left\{\frac{2}{\eta},2,\frac{1}{\rho c_0^2C_e^{(1)}}\right\}=\min\left\{\frac{1}{C_e^{(1)}\max_{i,j,k,\ell}\|B^{ijk\ell}\|_{\mathcal{C}^0(\overline{\Omega})}},2,\frac{1}{\rho c_0^2C_e^{(1)}}\right\},
	$$
	we obtain that another application of the Cauchy-Schwarz inequality (cf., e.g., Proposition~5.3 in~\cite{CanDAp14}) Young's inequality~\cite{Young1912}, Korn's inequality (Theorem~\ref{KornCart}) and the uniform positive-definiteness of the fourth order three-dimensional elasticity tensor $\{A^{ijk\ell}\}$ gives that for a.a. $T_0<t<T$:
	\begin{equation}
		\label{eq:equiv-2}
		\begin{aligned}
			&H_{\kappa,\varepsilon}(t)\ge E_\kappa(t)-\dfrac{\varepsilon\rho}{2}\|\dot{\bm{u}}_\kappa(t)\|_{\bm{L}^2(\Omega)}^2-\dfrac{\varepsilon\rho}{2}\|\bm{u}_\kappa(t)\|_{\bm{V}(\Omega)}^2\\
			&\ge\rho\left(1-\dfrac{\varepsilon}{2}\right)\int_{\Omega}|\dot{\bm{u}}_\kappa(t)|^2\dd x+\dfrac{1}{2}\left(1-\varepsilon\rho c_0^2C_e^{(1)}\right)\int_{\Omega}A^{ijk\ell}e_{k\|\ell}(\bm{u}_\kappa(t))e_{i\|j}(\bm{u}_\kappa(t))\dd x\\
			&\quad+\dfrac{1}{2\kappa}\int_{\Gamma}\left|\{[\bm{I}+u_{i,\kappa}(t)\bm{e}^i]\cdot\bm{q}\}^{-}\right|^2\dd\Gamma
			\ge\min\left\{1-\dfrac{\varepsilon}{2},1-\varepsilon\rho c_0^2C_e^{(1)}\right\}E_\kappa(t).
		\end{aligned}
	\end{equation}
	
	Combining~\eqref{eq:equiv-1} and~\eqref{eq:equiv-2} gives that $E_\kappa$ and $H_{\kappa,\varepsilon}$ are equivalent, where the equivalence constants are independent of $T_0\le t\le T$, namely:
	\begin{equation}
		\label{eq:equiv-3}
		\min\left\{1-\dfrac{\varepsilon}{2},1-\varepsilon\rho c_0^2C_e^{(1)}\right\}E_\kappa(t)\le H_{\kappa,\varepsilon}(t)\le\max\left\{1+\dfrac{\varepsilon}{2},1+\varepsilon\rho c_0^2C_e^{(1)}\right\}E_\kappa(t),
	\end{equation}
	for all $T_0\le t\le T$. Combining~\eqref{step2} with~\eqref{eq:equiv-3} gives
	\begin{equation*}
		\dot{H}_{\kappa,\varepsilon}(t)\le\dfrac{-\min\left\{\dfrac{1}{\rho}\left(\dfrac{1}{2C_e^{(2)}c_0^2}-\dfrac{\rho }{C_e^{(1)}\max_{i,j,k,\ell}\|B^{ijk\ell}\|_{\mathcal{C}^0(\overline{\Omega})}}\right),\dfrac{1}{C_e^{(1)}\max_{i,j,k,\ell}\|B^{ijk\ell}\|_{\mathcal{C}^0(\overline{\Omega})}}\right\}}{\max\left\{1+\dfrac{\varepsilon}{2},1+\varepsilon\rho c_0^2C_e^{(1)}\right\}}H_{\kappa,\varepsilon}(t),
	\end{equation*}
	for a.a. $T_0<t<T$. Letting
	\begin{equation*}
		c_{\textup{d}}:=\dfrac{\min\left\{\dfrac{1}{\rho}\left(\dfrac{1}{2C_e^{(2)}c_0^2}-\dfrac{\rho }{C_e^{(1)}\max_{i,j,k,\ell}\|B^{ijk\ell}\|_{\mathcal{C}^0(\overline{\Omega})}}\right),\dfrac{1}{C_e^{(1)}\max_{i,j,k,\ell}\|B^{ijk\ell}\|_{\mathcal{C}^0(\overline{\Omega})}}\right\}}{\max\left\{1+\dfrac{\varepsilon}{2},1+\varepsilon\rho c_0^2C_e^{(1)}\right\}},
	\end{equation*}
	we have that an application of the Gronwall inequality~\cite{GW} gives
	\begin{equation*}
		H_{\kappa,\varepsilon}(t)\le H_{\kappa,\varepsilon}(T_0)\exp(-c_{\textup{d}}(t-T_0)),\quad\textup{ for all }T_0\le t\le T,
	\end{equation*}
	and, besides, combining the latter with~\eqref{eq:equiv-3} gives:
	\begin{equation*}
		\min\left\{\dfrac{1}{2c_0^2C_e^{(1)}},\rho\right\}(\|\dot{\bm{u}}_\kappa(t)\|_{\bm{L}^2(\Omega)}^2+\|\bm{u}_\kappa(t)\|_{\bm{V}(\Omega)}^2)\le E_\kappa(t)\le \dfrac{\max\left\{1+\dfrac{\varepsilon}{2},1+\varepsilon\rho c_0^2C_e^{(1)}\right\}}{\min\left\{1-\dfrac{\varepsilon}{2},1-\varepsilon\rho c_0^2C_e^{(1)}\right\}}E_\kappa(T_0)\exp(-c_{\textup{d}}(t-T_0)),
	\end{equation*}
	for a.a. $T_0\le t\le T$. Therefore the conclusion follows by observing that $E_\kappa(T_0)$ is bounded independently of $\kappa$, the mapping~\eqref{eq:map} being bounded in $[T_0,T]$ independently of $\kappa$. Letting
	\begin{equation*}
		\tilde{C}:=\dfrac{\max\left\{1+\dfrac{\varepsilon}{2},1+\varepsilon\rho c_0^2C_e^{(1)}\right\}}{\left(\min\left\{\dfrac{1}{2c_0^2C_e^{(1)}},\rho\right\}\right)\left(\min\left\{1-\dfrac{\varepsilon}{2},1-\varepsilon\rho c_0^2C_e^{(1)}\right\}\right)}E_\kappa(T_0),
	\end{equation*}
	thus gives the sought conclusion and the proof is complete.
\end{proof}

Let us now discuss the smallness condition~\eqref{eq:smallness}. Observe that the position of the constants $C_e^{(1)}$ and $C_e^{(2)}$ as well as the direction of the inequality suggest that the elastic behaviour is dominant over the viscous one. This is true for materials of Kelvin-Voigt type (cf., e.g., \cite{Christensen12,Slonimsky67}) for which it is true that when the applied force is removed the body returns to its rest position and that the magnitude of the displacement decays exponentially.

\addtocontents{toc}{\protect\setcounter{tocdepth}{1}}

\section*{Conclusion and final remarks}

In this paper, we demonstrated the existence of solutions for a model that describes the motion of a three-dimensional linearly viscoelastic body constrained to remaining confined within a specific half-space.We proposed and justified a concept of solution for this model by studying the asymptotic behaviour of the minimisers of the classical energy of three-dimensional linearly viscoelastic bodies with a relaxation term reflecting the degree to which a particular deformation violates the given constraint.
We showed that the solutions for the penalized model converge to the solutions of a system of hyperbolic-like variational inequalities. Essential for establishing the obtained existence result is Lemma~\ref{lem:tr}, which asserts the compactness of the time-dependent version of the trace operator.

\section*{Declarations}

%

\subsection*{Acknowledgements}

The author is extremely grateful to the Anonymous Referees for their careful reading and for the suggested improvements.

\subsection*{Ethical Approval}

Not applicable.

\subsection*{Availability of Supporting Data}

Not applicable.

\subsection*{Competing Interests}

All authors certify that they have no affiliations with or involvement in any organisation or entity with any competing interests in the subject matter or materials discussed in this manuscript.

\subsection*{Funding}

This research was partly supported by the University Development Fund of The Chinese University of Hong Kong, Shenzhen.
	
	\bibliographystyle{abbrvnat} 
	\bibliography{references.bib}
	\newpage
	\renewcommand{\glossarysection}[1][]{}
	\printnoidxglossary[type=symbols,style=long,title={\scshape{List of Symbols}},toctitle={List of Symbols}]
\end{document}